\documentclass[10pt,final,twocolumn]{IEEEtran}
\long\def\comment#1{}

\usepackage{textcomp}
\usepackage{cite}
\usepackage{graphicx}
\usepackage{subfigure}
\usepackage{psfrag}
\usepackage{url}
\usepackage{stfloats}
\usepackage{amsmath}
\usepackage{amssymb,amsfonts}
\usepackage{amsthm}
\usepackage{float}
\usepackage[usenames]{color}
\usepackage{algorithm,algorithmic}
\usepackage{mathtools}

\newtheorem{corollary}{Corollary}
\newtheorem{assumption}{Assumption}
\newtheorem{remark}{Remark}

\newtheorem{lemma}{Lemma}
\newtheorem{theorem}{Theorem}

\begin{document}

\setlength{\arraycolsep}{0.3em}

\title{Distributed Online Optimization for Multi-Agent Networks with Coupled Inequality Constraints
\thanks{}}

%\author{Xiuxian Li, Xinlei Yi, and Lihua Xie
\author{Xiuxian Li, Xinlei Yi, and Lihua Xie
\thanks{This work was supported by the Ministry of Education of Singapore under MoE Tier 1 Research Grant RG72. Corresponding author: L. Xie.}
\thanks{X. Li and L. Xie are with School of Electrical and Electronic Engineering, Nanyang Technological University, 50 Nanyang Avenue, Singapore 639798 (e-mail: xxli@ieee.org; elhxie@ntu.edu.sg).}
\thanks{X. Yi is with the ACCESS Linnaeus Centre, Electrical Engineering, KTH Royal Institute of Technology, 100 44, Stockholm, Sweden (e-mail: xinleiy@kth.se).}
}

\maketitle

\setcounter{equation}{0}
\setcounter{figure}{0}
\setcounter{table}{0}
%%%%%%%%%%%%%%%%%%%%%%%%%%%%%%%%%%%%%%%%%%%%%%%%%%%%%%%%%%%%%%%%%%%%%%%%%%%%%%%%%%%%%%%%%

\begin{abstract}
This paper investigates the distributed online optimization problem over a multi-agent network subject to local set constraints and coupled inequality constraints, which has a lot of applications in many areas, such as wireless sensor networks, power systems and plug-in electric vehicles. In this problem, the cost function at each time step is the sum of local cost functions with each of them being gradually revealed to its corresponding agent, and meanwhile only local functions in coupled inequality constraints are accessible to each agent. To address this problem, a modified primal-dual algorithm, called distributed online primal-dual push-sum algorithm (DOPP), is developed in this paper, which does not rest on any assumption on parameter boundedness and is applicable to unbalanced networks. It is shown that the proposed algorithm is sublinear for both the dynamic regret and the violation of coupled inequality constraints. Finally, the theoretical results are supported by a simulation example.
\end{abstract}


\section{Introduction}\label{s1}

With the rapid development of advanced communication and computing technologies and low-cost devices, distributed optimization problems have recently attracted much attention from diverse communities, e.g., systems and control community, because a large number of practical problems boil down to distributed optimization problems over multi-agent networks, such as machine learning, statistical learning, sensor networks, resource allocation, formation control, and power systems \cite{bullo2009distributed,rabbat2004distributed,tsitsiklis1986distributed,li1987distributed,xu2017distributed}. Distinct from classic centralized optimization, distributed optimization involves multiple agents over a network which have their individual private information, and there exist no centralized agents that can access the entire information over the network. As such, an individual agent does not have adequate information to handle the optimization problem alone, and all agents need to exchange their local information in order to cooperatively solve the global optimization problem, see \cite{nedic2009distributed,aybat2018distributed,xu2018bregman,li2019distributed}, etc.

This paper focuses on distributed online optimization, which has numerous applications, such as prediction from expert advice, online spam filtering, online shortest paths, portfolio selection, and recommendation systems \cite{hazan2016introduction}. Note that online optimization was first investigated for centralized scenarios in machine learning community \cite{zinkevich2003online,hazan2007logarithmic,shalev2012online}. In centralized online optimization, there exists a sequence of time-dependent convex cost functions, which are not known as a priori knowledge and only revealed gradually. To be specific, the cost function at current time slot is accessible only after a decision is made at the current time. To measure the performance of an online algorithm, it is conventional to compare the accumulated cost associated with the sequential cost functions incurred by the algorithm at each time step with the cost incurred by the best fixed/dynamic decisions in hindsight, i.e., the minimal cost under the condition that all the cost functions at all times are known, and the metric, the difference between the two costs, is called static/dynamic regret. In general, an online algorithm is declared ``good'' if the regret is sublinear. For example, the author in \cite{zinkevich2003online} considered the online optimization problem subject to feasible set constraints, and an online subgradient projection algorithm was proposed. Later, the authors in \cite{hazan2007logarithmic} and \cite{shalev2012online} further addressed the same problem as in \cite{zinkevich2003online}. Recently, a sequence of time-varying inequality constraints have been treated for the online optimization in \cite{neely2017online,paternain2016online,chen2017online}.

Due to the emergence of complex tasks and big data in modern life, a single agent in general cannot acquire enough information to perform a complicated task because of its limited sensing and computational ability, etc. Hence, it is beneficial and preferable for a family of agents to accomplish an optimization task in a cooperative manner. As a consequence, recent years have witnessed a wide spectrum of research on distributed online optimization over multi-agent networks, such as \cite{mateos2014distributed,akbari2017distributed,nedic2015decentralized,koppel2015saddle,shahrampour2017online,shahrampour2018distributed,hosseini2016online,lee2017stochastic,
yuan2017adaptive,lee2017sublinear,paternain2019distributed,sharma2020distributed}, in which a collection of agents cooperatively deal with an online optimization problem. For example, distributed online optimization problems without constraints were considered in \cite{mateos2014distributed} by an online subgradient descent algorithm with the proportional-integral disagreement and in \cite{akbari2017distributed} by a distributed online subgradient push-sum algorithm. Also, distributed online optimization has been further studied under global/local set constraints, such as, a Nesterov based primal-dual algorithm \cite{nedic2015decentralized}, a variant of the Arrow-Hurwicz saddle point algorithm \cite{koppel2015saddle}, a mirror descent algorithm \cite{shahrampour2017online,shahrampour2018distributed}, a dual subgradient averaging algorithm \cite{hosseini2016online}, and a distributed primal-dual algorithm \cite{lee2017stochastic}. In addition, besides local feasible set constraints, local inequality constraints were considered in \cite{yuan2017adaptive} with the design of a consensus-based adaptive primal-dual subgradient algorithm. As an application of distributed online optimization, smart grids were discussed in \cite{zhou2017incentive}. More recently, a general constraint, i.e., a coupled inequality constraint, was investigated in \cite{lee2017sublinear} for distributed online optimization, where a distributed primal-dual algorithm is proposed and a sublinear static regret is achieved. It is known that coupled inequality constraints have a multitude of applications in optimal wireless networking \cite{lee2017sublinear}, smart grids, plug-in electric vehicles \cite{vujanic2016decomposition}, etc. It should be noted that coupled inequality constraints have been addressed for distributed optimization in \cite{chang2014distributed,mateos2017distributed,falsone2017dual,notarnicola2019constraint,notarnicola2017duality,li2019distributed4}, but \cite{lee2017sublinear} is the first one to consider distributed optimization with coupled inequality constraints in the online setup. However, \cite{lee2017sublinear} assumes the boundedness of Lagrange multipliers for its algorithm, which limits its applicability, since the multipliers are generated by the designed algorithm.

This paper revisits distributed online optimization subject to coupled inequality constraints, where all involved functions, including objective and constraint functions, are revealed gradually over time, and all agents are unaware of future information. To solve this problem, a different algorithm from \cite{lee2017sublinear} is proposed, which can achieve a sublinear dynamic regret under relaxed conditions. The contributions of this paper can be summarized as follows:
\begin{enumerate}
  \item In comparison with \cite{lee2017sublinear}, the results in this paper do not rely on the assumption that Lagrange multipliers generated by the proposed algorithm are bounded. Note that the removal of this assumption is nontrivial.
  \item Balanced communication graphs have been used for all agents' information exchange in \cite{lee2017sublinear}. In contrast, more general interaction graphs, i.e., unbalanced graphs, are considered in this paper for distributed online optimization. To cope with the imbalance of communication graphs, a push-sum idea \cite{kempe2003gossip,benezit2010weighted,dominguez2011distributed,tsianos2012push,nedic2015distributed,nedic2016stochastic,xi2017dextra} is exploited for designing our algorithm in order to counteract the effect of graph's imbalance.
  \item The dynamic regret is used for measuring the performance of the designed algorithm, which is shown to be sublinear when a weighted path variation is sublinear. Meanwhile, the metric for the violation of inequality constraints is also proved to be sublinear. Moreover, as a special case, the convergence speed for the time-invariant distributed optimization is provided.
\end{enumerate}

The rest of this paper is structured as follows. Section II presents some preliminary knowledge and formulates the considered problem. Section III provides the main results of this paper, and subsequently, a simulation example is provided for supporting the theoretical results in Section IV. Section V concludes this paper.

{\em Notations:} Denote by $[N]:=\{1,2,,\ldots,N\}$ the index set for a positive integer $N$. The set of $n$-dimensional vectors with nonnegative entries is denoted by $\mathbb{R}_+^n$. Let $col(z_1,\ldots,z_k)$ be the concatenated column vector of $z_i\in\mathbb{R}^n,i\in [k]$. Denote by $\|\cdot\|$ and $\|\cdot\|_1$ the standard Euclidean norm and $\ell_1$-norm, respectively. $x^\top$ and $\langle x,y\rangle$ denote the transpose of a vector $x$ and the standard inner product of $x,y\in\mathbb{R}^n$, respectively. Let $[z]_+$ be the component-wise projection of a vector $z\in\mathbb{R}^n$ onto $\mathbb{R}^n_+$ and $\{z\}_i$ be the $i$-th entry of $z$. Let $\mathbf{1},\mathbf{0}$ be the compatible column vectors of all entries $1$ and $0$, respectively. $I$ is the identity matrix of compatible dimension. Given two functions $h_1$ and $h_2$, $h_1=\Omega(h_2)$, $h_1=O_+(h_2)$, and $h_1=O(h_2)$ mean that there exist positive constants $C_1,C_2,C_3$ such that $h_1(x)\geq C_1 h_2(x)$, $h_1(x)\leq C_2 h_2(x)$, and $|h_1(x)|\leq C_3 h_2(x)$ for all $x$ in their domain, respectively. Denote by $\otimes$ the Kronecker product.

\section{Preliminaries}\label{s2}

\subsection{Graph Theory}\label{s2.1}

Denote by $\mathcal{G}_t=(\mathcal{V},\mathcal{E}_t)$ a simple graph at time slot $t$, where $\mathcal{V}=\{1,\ldots,N\}$ is the node set and $\mathcal{E}_t\subset\mathcal{V}\times\mathcal{V}$ is the edge set at time instant $t$. An edge $(j,i)\in\mathcal{E}_t$ means that node $j$ can route information to node $i$ at time step $t$, where $j$ is called an in-neighbor of $i$ and conversely, $i$ is called an out-neighbor of $j$. Denote by $\mathcal{N}_{i,t}^+=\{j:(j,i)\in\mathcal{E}_t\}$ and $\mathcal{N}_{i,t}^-=\{j:(i,j)\in\mathcal{E}_t\}$ the in-neighbor and out-neighbor sets of node $i$, respectively. It is assumed that $i\in\mathcal{N}_{i,t}^+$ and $i\in\mathcal{N}_{i,t}^-$ for all $i\in [N]$. The in-degree and out-degree of node $i$ at time $t$ are respectively defined by $d_{i,t}^+=|\mathcal{N}_{i,t}^+|$ and $d_{i,t}^-=|\mathcal{N}_{i,t}^-|$. A directed path is a sequence of directed consecutive edges, and a graph is called strongly connected if there is at least one directed path from any node to any other node in the graph. The adjacency matrix $A_t=(a_{ij,t})\in\mathbb{R}^{N\times N}$ at time $t$ is defined by: $a_{ij,t}>0$ if $(j,i)\in\mathcal{E}_t$, and $a_{ij,t}=0$ otherwise.

For the communication graph, the following standard assumptions (e.g., \cite{bullo2009distributed,nedic2009distributed,shahrampour2018distributed}) are imposed in this paper.
\begin{assumption}\label{a1}
For all $t\geq 0$, $\mathcal{G}_t$ satisfies:
\begin{enumerate}
  \item There exists a constant $0<a<1$ which lower bounds all nonzero weights, that is, $a_{ij,t}\geq a$ if $a_{ij,t}>0$, and $a_{ii,t}\geq a$ for all $i\in[N]$;
  \item The adjacency matrix $A_t$ is column-stochastic, i.e., $\sum_{i=1}^N a_{ij,t}=1$ for all $j\in [N]$;
  \item There exists a constant $Q>0$ such that the graph $(\mathcal{V},\cup_{l=0,\ldots,Q-1}\mathcal{E}_{t+l})$ is strongly connected for all $t\geq 1$.
\end{enumerate}
\end{assumption}

It is worth pointing out that Assumption \ref{a1} is less conservative than that in \cite{lee2017sublinear}, where $A_t$ is assumed to be doubly stochastic, i.e., the graph is balanced.

\subsection{Optimization Theory}\label{s2.2}

The {\em projection} of a point $x\in\mathbb{R}^n$ onto a closed convex set $S\subset\mathbb{R}^n$ is defined to be the point that has the shortest distance to $x$, that is, $P_{S}(x):=\arg\min_{y\in S}||x-y||$, satisfying
\begin{align}
&(x-P_S(x))^\top(y-P_S(x))\leq 0,~~~\forall x\in\mathbb{R}^n,~\forall y\in S       \label{pr1}\\
&\|P_S(z_1)-P_S(z_2)\|\leq \|z_1-z_2\|,~~~\forall z_1,z_2\in\mathbb{R}^n.        \label{pr2}
\end{align}

For a convex function $g:\mathbb{R}^n\to\mathbb{R}$, a {\em subgradient} of $g$ at a point $x\in\mathbb{R}^n$ is defined to be a vector $s\in\mathbb{R}^n$ such that
\begin{align}
g(y)-g(x)\geq s^\top(y-x),~\forall y\in\mathbb{R}^n,        \label{5}
\end{align}
and the set of subgradients at $x$ is called the {\em subdifferential} of $g$ at $x$, denoted by $\partial g(x)$. When the function $g$ is differentiable, the subdifferential at any point only has a single element, which is exactly the gradient, denoted by $\nabla g(x)$ at $x$.

A function $L:\Upsilon\times\Lambda\to\mathbb{R}$, where $\Upsilon\subset\mathbb{R}^n,\Lambda\subset\mathbb{R}^m$, is called {\em convex-concave} if $L(\cdot,\lambda):\Upsilon\to\mathbb{R}$ is convex for every $\lambda\in\Lambda$ and $L(x,\cdot):\Lambda\to\mathbb{R}$ is concave for each $x\in\Upsilon$. A {\em saddle point} of $L$ is defined to be a pair $(x^*,\lambda^*)$ such that
\begin{align}
L(x^*,\lambda)\leq L(x^*,\lambda^*)\leq L(x,\lambda^*),~~~\forall x\in\Upsilon,\lambda\in\Lambda.         \label{6}
\end{align}

Given an optimization problem
\begin{align}
\min_{x\in S} h_1(x),~~~~s.t.~~h_2(x)\leq \textbf{0},       \label{op1}
\end{align}
where $h_1(x):\mathbb{R}^n\to \mathbb{R}$ and $h_2(x):\mathbb{R}^n\to\mathbb{R}^m$ are convex functions, and $S\subset\mathbb{R}^n$ is a nonempty convex and closed set. In this paper an inequality or equality is understood componentwise. For (\ref{op1}), usually called the {\em primal} problem, the Lagrangian function is defined by $L(x,\mu)=h_1(x)+\mu^\top h_2(x)$, where $\mu$ is called the {\em dual variable} or {\em Lagrange multiplier} associated with the problem. Then, the {\em Lagrangian dual} problem is given as
\begin{align}
\max_{\mu\in \mathbb{R}_+^m} q(\mu),       \label{op3}
\end{align}
where $q(\mu):=\min_{x\in S}L(x,\mu)$, called {\em Lagrange dual function}. Let $h^*$ and $q^*$ be the optimal values of (\ref{op1}) and (\ref{op3}), respectively. As is known, the weak duality $q^*\leq h^*$ is always true, and the strong duality $q^*=h^*$ holds if a constraint qualification, such as Slater's condition, holds \cite{bertsekas2003convex,boyd2004convex,ruszczynski2006nonlinear}.

\subsection{Problem Formulation}\label{s2.3}

This section formulates the distributed online optimization problem. In this problem, there exist a sequence of time-varying global cost functions $\{f_t(x)\}_{t=0}^{\infty}$ which are not known in advance and only revealed gradually over time. At each time step $t$, the global cost function $f_t$ is composed of a group of local cost functions over a network with $N$ agents, i.e.,
\begin{align}
f_t(x)=\sum_{i=1}^N f_{i,t}(x_i),            \label{7a}
\end{align}
where $x:=col(x_1,\ldots,x_N)$ with $x_i\in X_i\subset\mathbb{R}^{n_i}$, and $f_{i,t}:\mathbb{R}^{n_i}\to\mathbb{R}$. After agent $i\in [N]$ makes a decision at time $t$, say $x_{i,t}$, the cost function $f_{i,t}$ is only revealed to agent $i$ and a cost $f_{i,t}(x_{i,t})$ is incurred. That is, each agent only gradually accesses the information of $f_{i,t}$ along with an incurred cost. In the meantime, there also exist a collection of functions $g_{i}:\mathbb{R}^{n_i}\to\mathbb{R}^m,i\in [N]$ which impose global and coupled inequality constraints for the online optimization problem, that is, at each time step $t$ it should satisfy
\begin{align}
g(x):=\sum_{i=1}^N g_{i}(x_i)\leq \textbf{0},           \label{7b}
\end{align}
where $g_{i}$ is only known to agent $i$ for each $i\in[N]$. For brevity, let $X=\times_{i=1}^N X_i$ be the Cartesian product of $X_i$'s, and define
\begin{align}
\mathcal{X}:=\{x\in X: g(x)\leq \textbf{0}\},        \label{8}
\end{align}
which is assumed nonempty.

The goal of the distributed online optimization is to reduce the total incurred cost over a finite time horizon $T>0$. Specifically, the aim is to design an algorithm such that
\begin{align}
Reg(T):=\sum_{t=1}^T\sum_{i=1}^N f_{i,t}(x_{i,t})-\sum_{t=1}^T\sum_{i=1}^N f_{i,t}(x_{i,t}^*)           \label{9}
\end{align}
is minimized, where (\ref{9}) is called the {\em dynamic regret} for measuring the performance of a designed algorithm, where $x_{i,t}^*$ is the $i$-th component of $x_t^*=col(x_1^*,\ldots,x_N^*)$ and
\begin{align}
x_t^*:=\mathop{\arg\min}_{x\in\mathcal{X}}\sum_{i=1}^N f_{i,t}(x_i),           \label{10}
\end{align}
that is, $x_t^*$ is the optimal decision vector at time step $t$. It is worth mentioning that another metric, called {\em static regret}, is defined by (\ref{9}) with $x_{i,t}^*$ being replaced with $x_i^*$, where $x^*=col(x_1^*,\ldots,x_N^*):=\mathop{\arg\min}_{x\in\mathcal{X}}\sum_{t=1}^T\sum_{i=1}^N f_{i,t}(x_i)$. That is, $x^*$ is the best decision vector by having the full knowledge of $f_{i,t},i\in[N],t\in[T]$ as an a priori and without any communication restrictions among agents. It is easy to observe that the static regret is not greater than the dynamic regret. Moreover, the dynamic regret makes more sense than the static one in many applications, such as tracking moving targets, where the variable of interest evolves over time and thus it is not sufficient to compare with a static benchmark. Note that the dynamic and static regrets will be identical when $f_{i,t}$'s are all independent of time $t$.

Generally speaking, a proposed algorithm is announced ``good'' if the regret is sublinear with respect to $T$, i.e., $Reg(T)=o(T)$, where $o(T)$ means that $\lim_{T\to\infty}o(T)/T=0$. Intuitively, the sublinearity of the regret guarantees that the average value of the global cost function over time horizon $T$ achieves the optimal value as $T$ goes to infinity.

Moreover, as the distributed online optimization involves coupled inequality constraints (\ref{7b}), it is indispensable for the designed algorithm to eventually respect this kind of constraints. That is, the following constraint violation
\begin{align}
Reg^c(T):=\Big\|\Big[\sum_{t=1}^T\sum_{i=1}^N g_{i}(x_{i,t})\Big]_+\Big\|           \label{11}
\end{align}
should grow more slowly than $T$. Mathematically, it should be ensured by the designed algorithm that $Reg^c(T)$ is also sublinear with respect to $T$, i.e., $Reg^c(T)=o(T)$.

To end this section, some necessary assumptions on the online optimization problem are listed as follows.

\begin{assumption}\label{a2}
\begin{enumerate}
  \item The functions $f_{i,t}$ and $g_{i}$ are convex on $\mathbb{R}^{n_i}$ for all $i\in[N]$ and $t\geq 0$.
  \item All the sets $X_i,i\in[N]$ are convex and compact.
  \item There exists a point $x_s\in relint(X)$ such that $g(x_s)\leq {\bf 0}$ for those components of $g(x)$ that are linear in $x$, if any, while $g(x_s)<{\bf 0}$ for all other components, where $relint(X)$ means the relative interior of $X$.
  \item Each $f_{i,t}$ and its subgradient are uniformly bounded, i.e., there exist $B_f,C_f>0$ such that $\forall t\geq 0,\forall i\in[N]$,
\begin{align}
|f_{i,t}(x)|&\leq B_f,~~~~~~~~~~~~\forall x\in X_i,               \label{13}\\
|f_{i,t}(x)-f_{i,t}(y)|&\leq C_f\|x-y\|,~~~\forall x,y\in X_i,                         \label{15}\\
\|\partial f_{i,t}(x)\|&\leq C_f,~~~~~~~~~~~~~\forall x\in X_i.                       \label{17a}
\end{align}
\end{enumerate}
\end{assumption}

The first assumption above does not require each function to be differentiable. The third assumption is the standard Slater's condition for the existence of saddle points in convex optimization problems \cite{ruszczynski2006nonlinear}. The second assumption has been widely employed in distributed online optimization \cite{nedic2015decentralized,lee2017stochastic,lee2017sublinear}, mostly due to the fact that decision variables are usually bounded in practice, such as the charging rate for the problem in Section \ref{s5}. The compactness of all $X_i$'s can result in that there exist positive constants $B_x$ and $B_g$ such that
\begin{align}
\|x\|&\leq B_x,~~~~~\forall x\in X_i,                                                            \label{12}\\
\|g_{i}(x)\|&\leq B_g,~~~~~\forall x\in X_i, \forall i\in[N].                    \label{14}
\end{align}
Furthermore, in light of the facts that $f_{i,t},g_{i}$ are convex and $X_i$'s are compact, it can be concluded that there exists $C_g>0$ such that for any $x,y\in X_i$ and $i\in[N],t\geq 0$,
\begin{align}
\|g_{i}(x)-g_{i}(y)\|&\leq C_g\|x-y\|,                        \label{16}\\
\|\partial g_{i}(x)\|&\leq C_g.                        \label{17b}
\end{align}

\section{Main Results}\label{s3}

This section presents the main results of this paper, including the algorithm design and the bounds on its regret and constraint violation. To start with, the Lagrangian function $L_t:\mathbb{R}^n\times \mathbb{R}_+^m\to\mathbb{R}$ of the online optimization problem at time instant $t$ is defined as
\begin{align}
L_t(x,\mu)&=\sum_{i=1}^N f_{i,t}(x_i)+\mu^\top\sum_{i=1}^N g_{i}(x_i),         \label{18}
\end{align}
where $n$ is the dimension of $x\in X$, i.e., $n:=\sum_{i=1}^N n_i$, and $\mu\geq \textbf{0}$ is the dual variable or Lagrange multiplier of this problem. By defining
\begin{align}
L_{i,t}(x_i,\mu):=f_{i,t}(x_i)+\mu^\top g_{i}(x_i),                           \label{19}
\end{align}
it is easy to see that $L_t(x,\mu)=\sum_{i=1}^N L_{i,t}(x_i,\mu)$.

For the centralized online optimization where only one centralized agent exists in the network and attempts to solve the optimization problem, a well-known algorithm is the so-called Arrow-Hurwicz-Uzawa saddle point algorithm or primal-dual algorithm \cite{arrow1958studies} by leveraging subgradients of primal and dual variables of the Lagrangian function $L_t$, explicitly given as
\begin{align}
x_{t+1}&=P_X(x_t-\alpha_t s_{x,t}),                         \nonumber\\
\mu_{t+1}&=[\mu_t+\alpha_t \nabla_\mu L_t(x_t,\mu_t)]_+,        \label{20}
\end{align}
where $\alpha_t$ is the stepsize, $\nabla_\mu L_t(x_t,\mu_t)=\sum_{i=1}^N g_{i}(x_{i,t})$, and $s_{x,t}$ is a subgradient of $L_t$ with respect to $x$ at $(x_t,\mu_t)$, i.e.,
\begin{align}
s_{x,t}\in \partial_x\Big(\sum_{i=1}^N f_{i,t}(x_{i,t})\Big)+\partial_x\Big(\sum_{i=1}^N g_{i}(x_{i,t})\Big)\mu_t.        \label{22}
\end{align}

However, in the scenario of distributed online optimization, no centralized agent can access the full knowledge of $f_t(x)$ and $g(x)$, which are only gradually revealed to each individual agent in the network. Hence, algorithm (\ref{20}) is not applicable directly since each agent does not have an identical $\mu_t$ and does not know $\nabla_\mu L_t(x_t,\mu_t)$ at time $t$. As such, the authors in \cite{lee2017sublinear} proposed a modified algorithm based on (\ref{20}), i.e.,
\begin{align}
x_{i,t+1}&=P_{X_i}(x_{i,t}-\alpha_t s_{i,t}'),            \nonumber\\
\mu_{i,t+1}&=\Big[\sum_{j=1}^N a_{ij,t}\mu_{j,t}+\alpha_t \sum_{j=1}^N a_{ij,t}y_{j,t}\Big]_+,        \nonumber\\
y_{i,t+1}&=\sum_{j=1}^N a_{ij,t}y_{j,t}+g_{i}(x_{i,t+1})-g_{i}(x_{i,t}),                 \label{23}
\end{align}
where $s_{i,t}'\in \partial f_{i,t}(x_{i,t})+\partial g_{i}(x_{i,t})\sum_{j=1}^N a_{ij,t}\mu_{j,t}$, and $y_{i,t}$ is an auxiliary variable of agent $i$ for tracking the function $\sum_{i=1}^N g_{i}(x_{i,t})/N$. It is shown that algorithm (\ref{23}) can ensure the sublinearity of both the regret and constraint violation. Nevertheless, (\ref{23}) builds upon an assumption that $\mu_{i,t}$'s are bounded for all $i\in[N]$ and $t\geq 1$, which limits its applicability since $\mu_{i,t}$ is generated by the algorithm (\ref{23}) and their boundedness should be theoretically established rather than by an assumption. On the other hand, algorithm (\ref{23}) is designed for balanced communication graphs among agents, yet not applicable for unbalanced interaction graphs which are more general and practical in applications. Note that under an unbalanced graph, such as $A_t=(a_{ij,t})$ is column-stochastic (but not row-stochastic), the multipliers $\mu_{i,t}$'s and thus $\sum_{j=1}^N a_{ij,t}\mu_{j,t}$ in $s_{i,t}'$ will eventually achieve different vector values, meaning that $\sum_{j=1}^N a_{ij,t}\mu_{j,t}$, as multipliers in the Lagrangian function, cannot reach an identical multiplier as they should. In this case, even when $f_{i,t}$'s are time-invariant, algorithm (\ref{23}) cannot converge to the optimizer set. Please refer to \cite{nedic2016stochastic,xie2018distributed} for more details.

As pointed out above, two challenges appear in this paper when handling problem (\ref{7a})-(\ref{7b}): one is to consider unbalanced communication graphs, as shown in Assumption \ref{a1}.2 for $A_t$, and the other is to eliminate the assumption on the boundedness of $\mu_{i,t}$ for all $i\in[N]$ and $t\geq 1$. To address the two issues, two strategies are respectively introduced in the sequel.

Firstly, to deal with unbalanced communication graphs, there are generally four methods which are respectively the push-sum method \cite{kempe2003gossip,benezit2010weighted,dominguez2011distributed,tsianos2012push,nedic2015distributed,nedic2016stochastic,xi2017dextra}, the ``surplus''-based method \cite{xi2017distributed}, the row-stochastic matrix method \cite{xi2018linear}, and the epigraph method \cite{xie2018distributed}. Among them, the push-sum approach, originally devised for average consensus problems over unbalanced graphs \cite{kempe2003gossip,benezit2010weighted,dominguez2011distributed}, is most popular. For the other three methods, there are some shortcomings. Specifically, the ``surplus''-based idea used in \cite{xi2017distributed} is required to access global information since a parameter in the algorithm depends on communication weight matrices, while some network-size variables are introduced for each agent in \cite{xi2018linear,xie2018distributed} which will incur extremely high computational complexity especially for large-scale networks. Based on the aforementioned discussion, in this paper we adopt the push-sum approach to handle the imbalance of the communication graph among agents. Note that it is reasonable for each agent in the push-sum method to know its own out-degree \cite{hendrickx2015fundamental}. Actually, as pointed out in \cite{nedic2016stochastic}, the information on the out-degree for each individual agent can be known by virtue of bidirectional exchange of ``hello'' messages during only a single round of communication. Specifically, in view of the push-sum idea, algorithm (\ref{23}) is redesigned as
\begin{align}
w_{i,t+1}&=\sum_{j=1}^N a_{ij,t} w_{j,t},              \nonumber\\
\hat{\mu}_{i,t}&=\sum_{j=1}^N a_{ij,t}\mu_{j,t},~~~\hat{y}_{i,t}=\sum_{j=1}^N a_{ij,t} y_{j,t},           \nonumber\\
x_{i,t+1}&=P_{X_i}(x_{i,t}-\alpha_t s_{i,t+1}),            \nonumber\\
\mu_{i,t+1}&=\Big[\hat{\mu}_{i,t}+\alpha_t \frac{\hat{y}_{i,t}}{w_{i,t+1}}\Big]_+,        \nonumber\\
y_{i,t+1}&=\hat{y}_{i,t}+g_{i}(x_{i,t+1})-g_{i}(x_{i,t}),                 \label{24}
\end{align}
where $s_{i,t+1}\in \partial f_{i,t}(x_{i,t})+\partial g_{i}(x_{i,t})\hat{\mu}_{i,t}/w_{i,t+1}$, and $w_{i,t}\in\mathbb{R}$ is a variable, aiming to remove the imbalance of the communication graph by, roughly speaking, tracking the right-hand eigenvector of $A_t$ associated with the eigenvalue $\textbf{1}$.

\begin{algorithm}
 \caption{Distributed Online Primal-dual Push-sum (DOPP)}
 \begin{algorithmic}[1]
 \renewcommand{\algorithmicrequire}{\textbf{Require:}}
 \renewcommand{\algorithmicensure}{\textbf{Output:}}
 \REQUIRE  Set $T\geq 4$. Locally initialize $w_{i,0}=1$, $x_{i,0}\in X_i$, $\mu_{i,0}=0$ and $y_{i,0}=g_{i}(x_{i,0})$ for all $i\in[N]$.
% \ENSURE  out
% \\ \textit{Initialization} :
  \STATE If $t=T$, then stop. Otherwise, update for each $i\in[N]$:
\begin{align}
w_{i,t+1}&=\sum_{j=1}^N a_{ij,t} w_{j,t},              \label{ag1}\\
\hat{\mu}_{i,t}&=\sum_{j=1}^N a_{ij,t}\mu_{j,t},~~~~~\hat{y}_{i,t}=\sum_{j=1}^N a_{ij,t} y_{j,t},           \label{ag2}\\
s_{i,t+1}&\in \partial f_{i,t}(x_{i,t})+\partial g_{i}(x_{i,t})\frac{\hat{\mu}_{i,t}}{w_{i,t+1}},      \label{25}\\
x_{i,t+1}&=P_{X_i}(x_{i,t}-\alpha_t s_{i,t+1}),            \label{ag3}\\
\mu_{i,t+1}&=\Big[\hat{\mu}_{i,t}+\alpha_t \Big(\frac{\hat{y}_{i,t}}{w_{i,t+1}}-\beta_t\hat{\mu}_{i,t}\Big)\Big]_+,        \label{ag4}\\
y_{i,t+1}&=\hat{y}_{i,t}+g_{i}(x_{i,t+1})-g_{i}(x_{i,t}),                 \label{ag5}
\end{align}
  \STATE Increase $t$ by one and go to Step 1.
 \end{algorithmic}
\end{algorithm}

Secondly, there is no guarantee on the boundedness of $\mu_{i,t}$ in algorithm (\ref{24}), as in algorithm (\ref{23}). To hinder the increase of a parameter, a quintessential method is to append some penalty function or term \cite{yuan2017adaptive,bertsekas2003convex,boyd2004convex,ruszczynski2006nonlinear}, inspired by which an additional penalty term is designed and incorporated into the update of $\mu_{i,t+1}$ in order to impede the growth of $\mu_{i,t}$, that is,
\begin{align}
\mu_{i,t+1}=\Big[\hat{\mu}_{i,t}+\alpha_t \Big(\frac{\hat{y}_{i,t}}{w_{i,t+1}}-\beta_t\hat{\mu}_{i,t}\Big)\Big]_+,                 \label{26}
\end{align}
where $\beta_t$ is a stepsize to be determined. Note that there is another method to handle the boundedness of $\mu_{i,t}$, that is, performing projections on some bounded set $M_i$ for agent $i$, instead of on $\mathbb{R}_+^m$, when updating $\mu_{i,t}$ at each time slot, as done in \cite{chang2014distributed,mateos2017distributed,li2019distributed4,zhu2012distributed}. However, the computation of the set $M_i$ is usually difficult and computationally expensive for distributed online optimization.

The proposed algorithm in this paper is summarized in Algorithm 1.

With the above preparations, it is now ready to present the main results of this paper.

\begin{theorem}\label{t1}
Under Assumptions \ref{a1} and \ref{a2}, and let $\alpha_0=1,\beta_0=1$, and for $t\geq 1$,
\begin{align}
\alpha_t=\frac{1}{\sqrt{t}},~~~~~\beta_{t}=\frac{1}{t^{\kappa}},        \label{28}
\end{align}
where $\kappa$ is a constant satisfying $\kappa\in(0,1/4)$, then the dynamic regret (\ref{9}) and constraint violation (\ref{11}) can be bounded as
\begin{align}
Reg(T)&=O_+(T^{\frac{1}{2}+2\kappa})+O_+(V_T),         \label{29}\\
Reg^c(T)&=O(T^{1-\frac{\kappa}{2}}),          \label{30}
\end{align}
where $V_T$ represents the $1/\alpha_t$-weighted path variation of the optimal decision vectors $x_{i,t}^*$'s, defined by
\begin{align}
V_T:=\sum_{t=1}^T\frac{1}{\alpha_t}\sum_{i=1}^N\|x_{i,t+1}^*-x_{i,t}^*\|.         \label{VT}
\end{align}
Moreover, in the worst case when $x_{t}:=col(x_{1,t},\ldots,x_{N,t})$ is always infeasible, i.e., $\sum_{i=1}^N g_i(x_{i,t})>\textbf{0}$ for all $t\in[T]$, then
\begin{align}
Reg(T)=\Omega(-T^{1-\frac{\kappa}{2}}).         \label{29b}
\end{align}
\end{theorem}
\begin{proof}
The proof can be found in Appendix B.
\end{proof}

\begin{remark}\label{r2}
It can be found from Theorem \ref{t1} that $Reg(T)$ has an upper bound which is close to $O(T^{1/2})$ when $\kappa$ is sufficiently small, and meanwhile $Reg^c(T)$ will reach a good upper bound when $\kappa$ is large enough. As a result, there should be a tradeoff for choosing $\kappa$ such that both $Reg(T)$ and $Reg^c(T)$ get good upper bounds. Simultaneously, the upper bound on $Reg(T)$ also depends on $V_T$, indicating that $Reg(T)$ is sublinear if $V_T$ is sublinear, which is reasonable for the dynamic regret analysis because it is impossible to track the optimal decision vectors $\{x_t^*\}_{t=0}^\infty$ when $\{x_t^*\}_{t=0}^\infty$ change dramatically (a widely known phenomenon in online optimization \cite{yang2016tracking,cao2018online}). It should be noted that $V_T=0$ if the static regret is studied. In comparison with \cite{lee2017sublinear}, where the same problem as (\ref{7a})-(\ref{7b}) has been studied, the sublinearity of $Reg(T)$ and $Reg^c(T)$ in Theorem \ref{t1} is obtained under less conservative assumptions, that is, no assumptions on boundedness of $\mu_{i,t}$ are employed here while it is utilized in \cite{lee2017sublinear}. In addition, the static regret and balanced communication graphs are considered in \cite{lee2017sublinear}, while the dynamic regret and more general unbalanced interaction graphs are taken into account here.
\end{remark}

\begin{remark}\label{rrr1}
More specifically, as seen from Appendix B, (\ref{29}), (\ref{30}) and (\ref{29b}) can be respectively established as $Reg(T)=O_+(\zeta T^{\frac{1}{2}+2\kappa})+O_+(V_T)$, $Reg^c(T)=O(N\sqrt{N\zeta}T^{1-\frac{\kappa}{2}})$, $Reg(T)=\Omega(-N\sqrt{N\zeta}T^{1-\frac{\kappa}{2}})$, where $\zeta:=\zeta_0+\max\{N^7/[a^2r^6(1-\lambda)^2],N^6/[r^7(1-\lambda)^3]\}$ with $r\geq 1/N^{NQ}$ and $\lambda\leq (1-1/N^{NQ})^{1/(NQ)}$ as defined in Lemma \ref{l1}, and $\zeta_0>0$ is a constant independent of $a,r,\lambda,Q$. It is noteworthy that $\zeta$ in Theorem \ref{t1} is proportional to $Q$, and thus the regret bound will increase as $Q$ grows. Moreover, in the case that $\mu_{i,t}$'s are assumed to be bounded as in \cite{lee2017sublinear}, the improved bound $O(\sqrt{T})$, which is the same as in \cite{lee2017sublinear}, can be established for the static regret and constraint violation for DOPP here under unbalanced graphs.
\end{remark}

\begin{remark}\label{rrr2}
Note that true subgradients have been leveraged in Theorem \ref{t1}. When only noisy subgradients are available, i.e., $\tilde{\partial}f_{i,t}(x_{i,t})=\partial f_{i,t}(x_{i,t})+\epsilon_{i,t}^f$ and $\tilde{\partial}g_{i}(x_{i,t})=\partial g_{i}(x_{i,t})+\epsilon_{i,t}^g$ with $\epsilon_{i,t}^f, \epsilon_{i,t}^g$ being i.i.d. unbiased noises and having bounded variances, the same bound for $\mathbb{E}(Reg(T))$ and $\mathbb{E}(Reg^c(T))$ as in Theorem \ref{t1} can be established using similar arguments, where $\mathbb{E}(\cdot)$ denotes the mathematical expectation.
\end{remark}

As discussed in Remark \ref{r2}, the parameter $\kappa$ can be specified for the same upper bound for $Reg(T)$ and $Reg^c(T)$ as follows.

\begin{corollary}\label{c1}
In Theorem \ref{t1}, let $\kappa=1/5$, then the regret (\ref{9}) and constraint violation (\ref{11}) can be upper bounded as
\begin{align}
Reg(T)&=O_+(T^{\frac{9}{10}})+O_+(V_T),         \label{32}\\
Reg^c(T)&=O(T^{\frac{9}{10}}).       \label{33}
\end{align}
\end{corollary}

\begin{proof}
To achieve the same upper bound for $Reg(T)$ and $Reg^c(T)$, it amounts to that $\frac{1}{2}+2\kappa=1-\frac{\kappa}{2}$, thus leading to $\kappa=1/5$, which directly implies (\ref{32}) and (\ref{33}).
\end{proof}

\begin{remark}\label{rrr3}
Note that distributed algorithms are considered here to achieve sublinearity for the dynamic regret under unbalanced graphs. In comparison, the static regret is studied in most of existing works without inequality constraints, and the dynamic regret was recently investigated in \cite{yi2019distributed} with time-varying coupled inequality constraints, where the bounds $O(\sqrt{T})+O(\sqrt{T}\sum_{t=1}^{T}\|x_{i,t+1}^*-x_{i,t}^*\|)$ and $O(\sqrt{T})$ are obtained for the dynamic regret and constraint violation, respectively. However, the results only applies to balanced graphs and the bound $\sqrt{T}\sum_{t=1}^{T}\|x_{i,t+1}^*-x_{i,t}^*\|$ is weaker than $V_T$ here, i.e., $V_T\leq \sqrt{T}\sum_{t=1}^{T}\|x_{i,t+1}^*-x_{i,t}^*\|$. Additionally, for general convex functions, it is known that the optimal bound for dynamic regret is $O(\sqrt{T})+O(\sqrt{T^{1-\tau}}\sqrt{\sum_{t=1}^{T-1}t^\tau\|x_{i,t+1}^*-x_{i,t}^*\|})$ for any pre-defined $\tau\in[0,1)$ in centralized online optimization without inequality constraints \cite{zhao2018proximal}. In this respect, how to establish the optimal bounds for the dynamic regret and constraint violation in our setting as in the centralized case is one potential research direction.
\end{remark}

As a special case of problem (\ref{7a})-(\ref{7b}), the time-invariant online optimization problem, that is, $f_{i,t}(x)$'s are independent of time $t$ for all $i\in[N]$ and are simply denoted by $f_i(x)$, can enjoy a better result, as shown below.

\begin{theorem}\label{t2}
For the time-invariant online optimization problem, if Assumptions \ref{a1} and \ref{a2} hold, and let $\alpha_0=1,\beta_0=1$, and for $t\geq 1$, $\alpha_t,\beta_t$ are given as in (\ref{28}), then
\begin{align}
&\sum_{i=1}^Nf_i(\bar{x}_{i,T})-\sum_{i=1}^N f_i(x_i^*)=O_+(T^{-\frac{1}{2}+2\kappa}),                                   \label{35}\\
&E_T^f=\left\{
                                                           \begin{array}{ll}
                                                             \Omega(-T^{-\kappa}), & \kappa\in(0,\frac{1}{6}] \\
                                                             \Omega(-T^{-\frac{1}{4}+\frac{\kappa}{2}}), & \kappa\in[\frac{1}{6},\frac{1}{4})
                                                           \end{array}
                                                         \right.             \label{35b}\\
&\Big\|\Big[\sum_{i=1}^N g_i(\bar{x}_{i,T})\Big]_+\Big\|=\left\{
                                                           \begin{array}{ll}
                                                             O(T^{-\kappa}), & \kappa\in(0,\frac{1}{6}] \\
                                                             O(T^{-\frac{1}{4}+\frac{\kappa}{2}}), & \kappa\in[\frac{1}{6},\frac{1}{4})
                                                           \end{array}
                                                         \right.                      \label{36}
\end{align}
where $E_T^f:=\frac{1}{T}\sum_{t=1}^T\sum_{i=1}^Nf_i(x_{i,t})-\sum_{i=1}^N f_i(x_i^*)$, and
\begin{align}
&\hspace{0.5cm}\bar{x}_{i,T}:=\frac{1}{T}\sum_{t=1}^T x_{i,t},~~~~~\forall i\in[N]             \label{95a}\\
&x^*=col(x_1^*,\ldots,x_N^*):=\mathop{\arg\min}_{x\in\mathcal{X}} \sum_{i=1}^N f_i(x_i).          \label{95b}
\end{align}
\end{theorem}
\begin{proof}
The proof can be found in Appendix C.
\end{proof}

\begin{corollary}\label{c2}
In Theorem \ref{t2}, let $\kappa=1/6$, then
\begin{align}
\sum_{i=1}^Nf_i(\bar{x}_{i,T})-\sum_{i=1}^N f_i(x_i^*)&=O_+(T^{-\frac{1}{6}}),                                   \label{c35}\\
\frac{1}{T}\sum_{t=1}^T\sum_{i=1}^Nf_i(x_{i,t})-\sum_{i=1}^N f_i(x_i^*)&=\Omega(-T^{-\frac{1}{6}}),             \label{c35b}\\
\Big\|\Big[\sum_{i=1}^N g_i(\bar{x}_{i,T})\Big]_+\Big\|&=O(T^{-\frac{1}{6}}).                                  \label{c36}
\end{align}
\end{corollary}

\begin{remark}\label{r3}
It should be noted that the distributed optimization problem studied in Theorem \ref{t2} has also been addressed in \cite{chang2014distributed,mateos2017distributed,falsone2017dual,notarnicola2019constraint,notarnicola2017duality,li2019distributed4}, where weight-balanced graphs are considered and the convergence rate $O(1/\sqrt{T})$ is provided in \cite{mateos2017distributed,li2019distributed4} in terms of the Lagrangian function. In comparison, the results in Theorem \ref{t2} are under unbalanced graphs, and establish the convergence speed for the cost function and constraint functions separately, as shown in (\ref{35})-(\ref{36}), although a slower rate $O(1/T^{1/6})$ is established here, as seen in Corollary \ref{c2}.
\end{remark}

\section{A Simulation Example}\label{s5}

This section applies Algorithm 1 to the Plug-in Electric Vehicles (PEVs) charging problem \cite{vujanic2016decomposition,falsone2017dual} in order to corroborate the algorithm's efficiency. The purpose of this PEVs charging problem is to seek an optimal overnight charging schedule for a collection of vehicles subject to some practical constraints, such as the limited charging rate for each vehicle and the overall maximal power that can be delivered by the whole network, etc.

As done in \cite{falsone2017dual}, a slightly modified ``only charging'' problem in \cite{vujanic2016decomposition} is taken into account here. That is, the charging rate of each vehicle is permitted to be optimized at each time step, rather than making a decision on whether or not to charge the vehicle at some fixed charging rate. Formally, the charging problem at time slot $t$ can be cast as $f_{i,t}(x_i)=c_{i,t}^\top x_i$ in (\ref{7a}) and $g_{i}(x_i)=D_ix_i-b/N$ in (\ref{7b}) with $x_i\in X_i\subset\mathbb{R}^{n_i}$ being a local feasible set constraint for each $i\in[N]$, where $X_i$ is usually a compact convex polygon in the charging problem. In this problem, the variable $x_i$ stands for the charging rate in a specified time duration, and $c_{i,t}$ represents the unitary charging cost (bounded) of vehicle $i$ at time instant $t$, randomly chosen in $[0,10]$ in the simulation. Also, $\sum_{i=1}^N (D_ix_i-b/N)\leq \textbf{0}$ is the coupled inequality constraint, representing the whole networked power constraint, where $D_i\in\mathbb{R}^{m\times n_i}$ is the rate aggregation matrix for each $i\in[N]$ and $b\in\mathbb{R}^m$ is the limit on the global aggregate charging power flow. Please refer to \cite{vujanic2016decomposition} for more details on the PEVs charging problem.

\begin{figure}[H]
\centering
\subfigure[]{\includegraphics[width=1.5in]{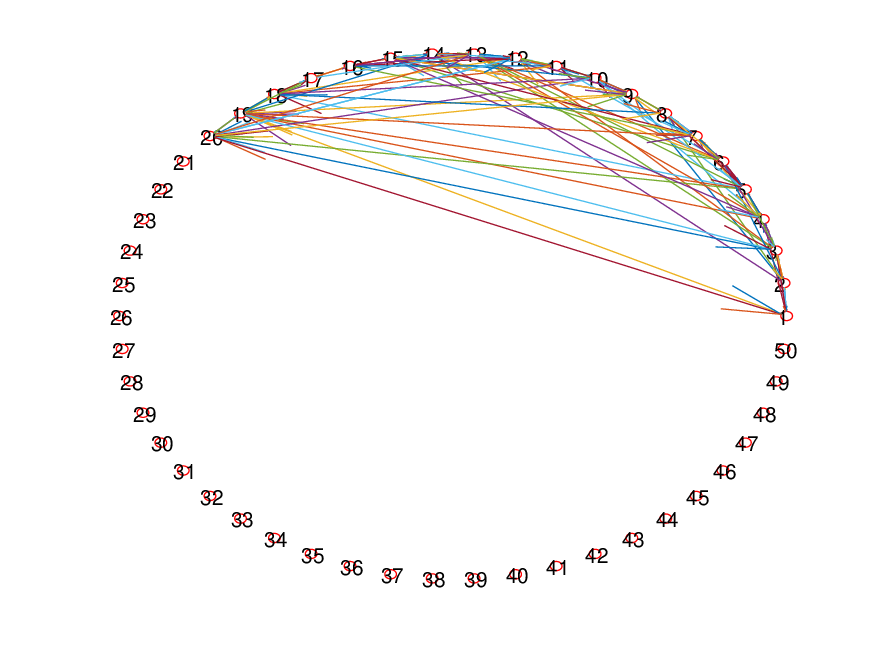}}
\subfigure[]{\includegraphics[width=1.5in]{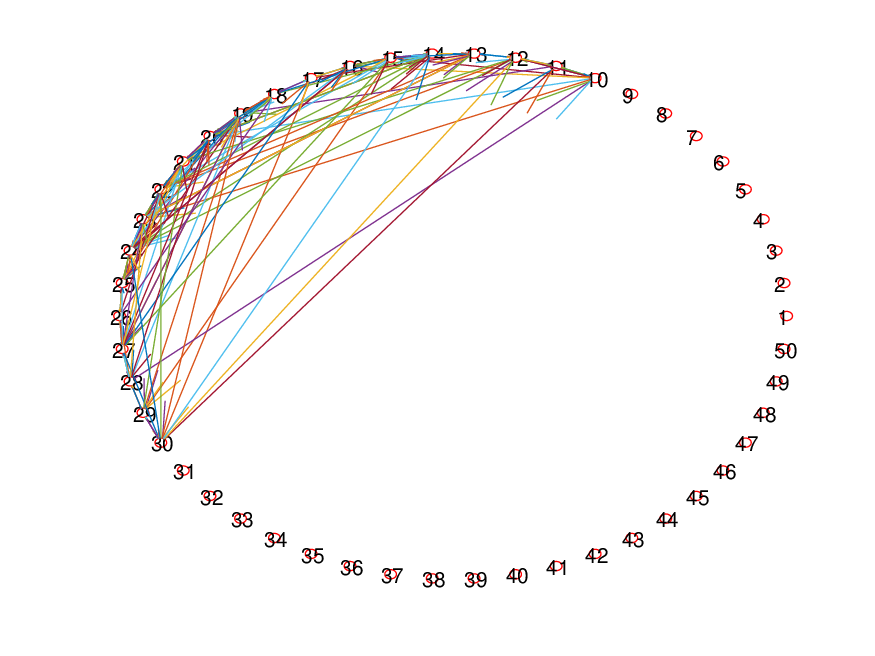}}

\subfigure[]{\includegraphics[width=1.5in]{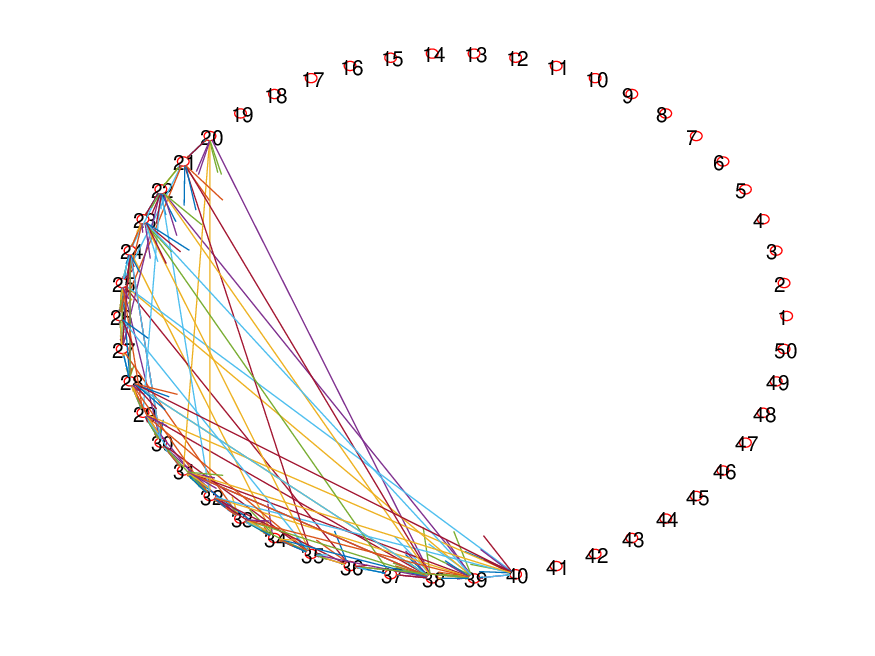}}
\subfigure[]{\includegraphics[width=1.5in]{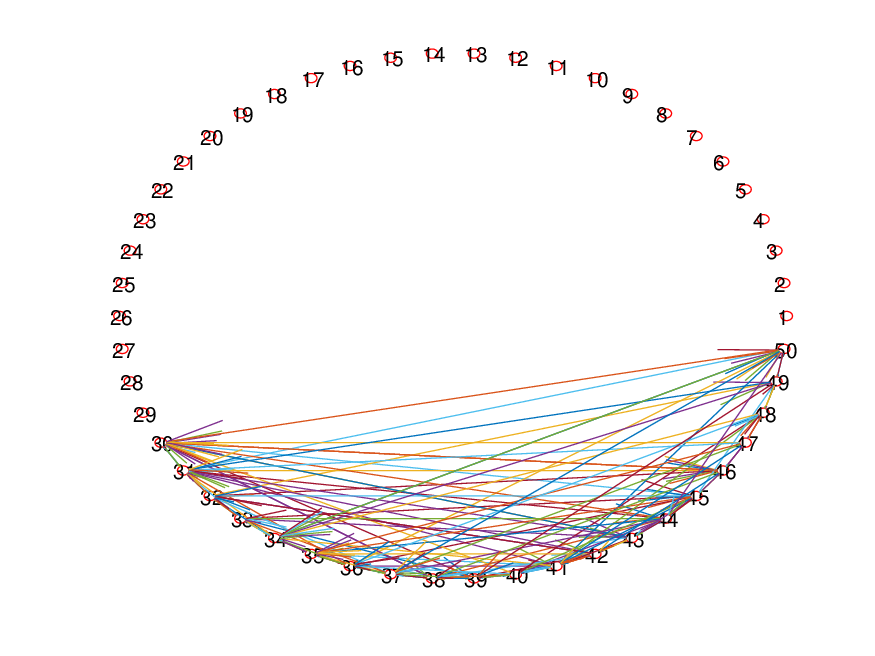}}
\caption{Schematic illustration of $4$ switching graphs.}
\label{f1}
\end{figure}

\begin{figure}[H]
\centering
\includegraphics[width=2.9in]{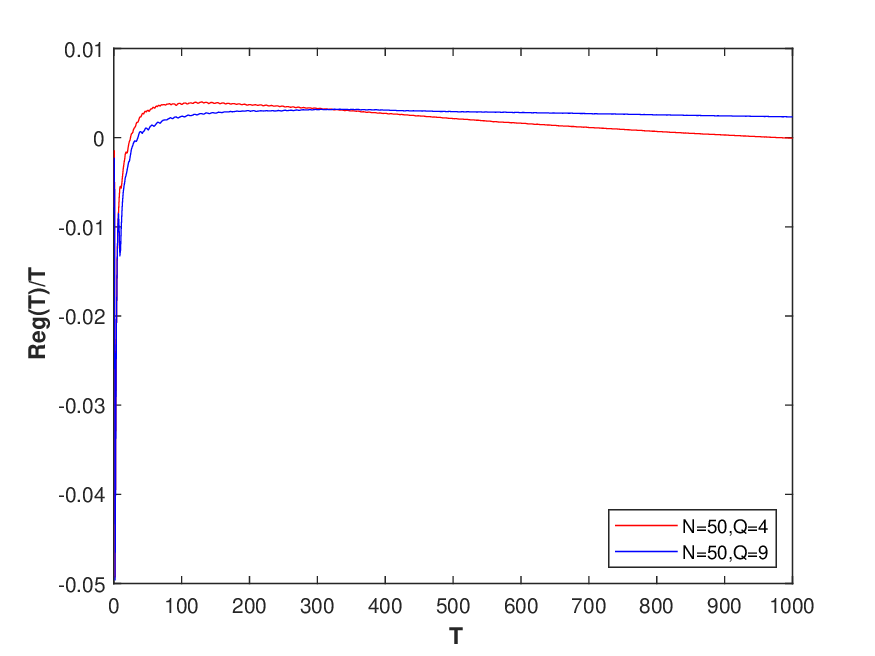}
\caption{Evolutions of $Reg(T)/T$ with $Q=4$ and $Q=9$ for $N=50$.}
\label{f2}
\end{figure}

For the charging problem, it is easy to verify that Assumption \ref{a2} holds based on the above facts. As given in \cite{vujanic2016decomposition,falsone2017dual}, the dimension of $x_i$ for each individual agent is $n_i=24$, each local feasible set $X_i$ is confined by $197$ inequalities, and the number of inequality constraints is $m=48$. In this setup, let $\kappa=0.2$, and different switching graphs are considered in this simulation along with the distinct number of agents. Specifically, Figs. \ref{f2} and \ref{f3} show the evolutions of $Reg(T)/T$ and $Reg^c(T)/T$ for a group of $N=50$ vehicles when $Q=4$ and $Q=9$, respectively, in which the trajectories are tending to the origin, supporting Algorithm 1. Note that $Q$ is given in Assumption \ref{a1} for communication graphs, and for instance, four switching graphs in Fig. \ref{f1} are employed here when $Q=4$. It is worthwhile to notice that the value of $Reg(T)/T$ in Fig. \ref{f2} can be negative, which is reasonable because the inequality constraints are not always respected by $x_{i,t}$. In addition, Figs. \ref{f5} and \ref{f6} give the trajectories of $Reg(T)/T$ and $Reg^c(T)/T$ for a fixed communication graph, i.e., $Q=1$, when $N=50$ and $N=100$, respectively, indicating the convergence of Algorithm 1 in this scenario. Besides, observing Figs. \ref{f4} and \ref{f7}, one can find that all $\tilde{\mu}_{i,t}$'s can achieve consensus asymptotically in these simulations.

\begin{figure}[H]
\centering
\includegraphics[width=2.9in]{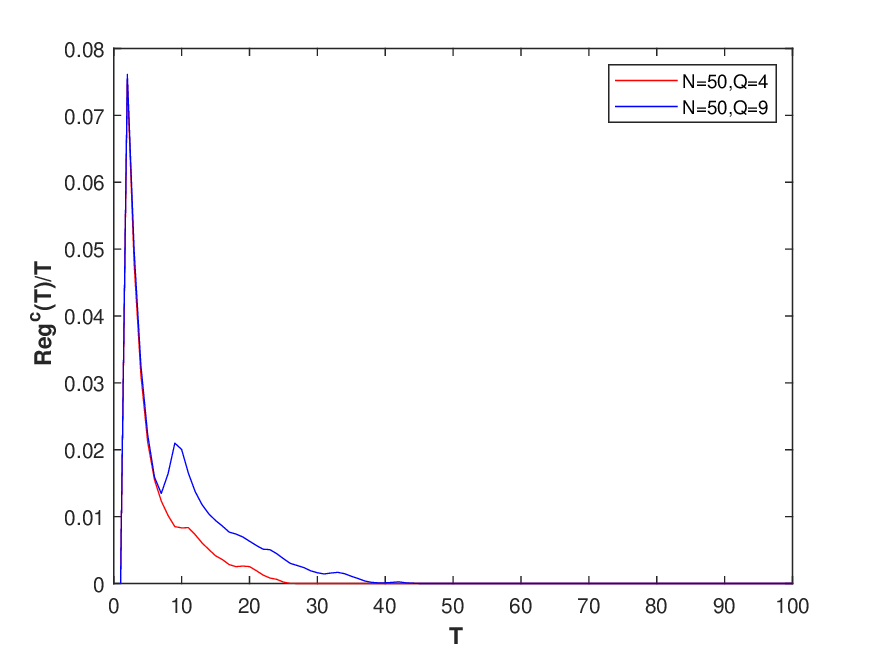}
\caption{Evolutions of $Reg^c(T)/T$ with $Q=4$ and $Q=9$ for $N=50$.}
\label{f3}
\end{figure}

\begin{figure}[H]
\centering
\includegraphics[width=2.9in]{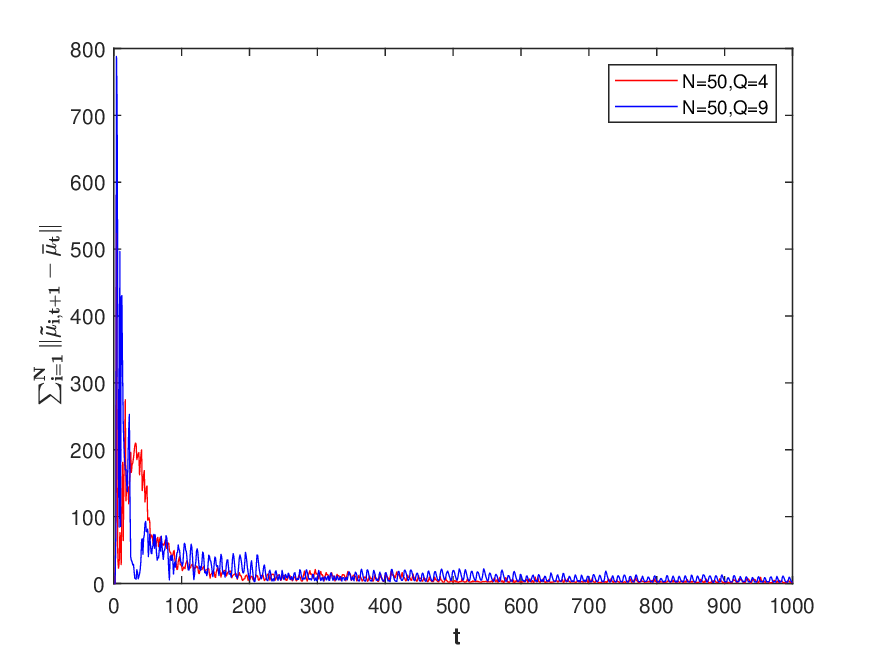}
\caption{Evolution of $\sum_{i=1}^N\|\tilde{\mu}_{i,t+1}-\bar{\mu}_t\|$ over $N$ agents.}
\label{f4}
\end{figure}

\section{Conclusion}\label{s6}

This paper has investigated distributed online convex optimization problems over directed multi-agent networks subject to local set constraints and coupled inequality constraints. It is noted that the same problem has been studied in \cite{lee2017sublinear} along with the design of an online primal-dual algorithm. However, the results in \cite{lee2017sublinear} depend on the boundedness of Lagrange multipliers generated by the proposed algorithm, which limits its applicability. To tackle this problem, a modified distributed online primal-dual push-sum algorithm (DOPP) has been proposed, which has been proven to possess the sublinear dynamic regret and constraint violation when a weighted path variation of optimal decision variables is sublinear. Moreover, unbalanced communication graphs have been considered for networked agents, which are more general. Finally, the algorithm's performance has been demonstrated by a numerical application. Future work can focus on further improving the convergence rate on $Reg(T)/T$ and $Reg^c(T)/T$.

\section*{Acknowledgment}

The authors are grateful to the Editor, the Associate Editor and the anonymous reviewers for their insightful suggestions.

\begin{figure}[H]
\centering
\includegraphics[width=2.9in]{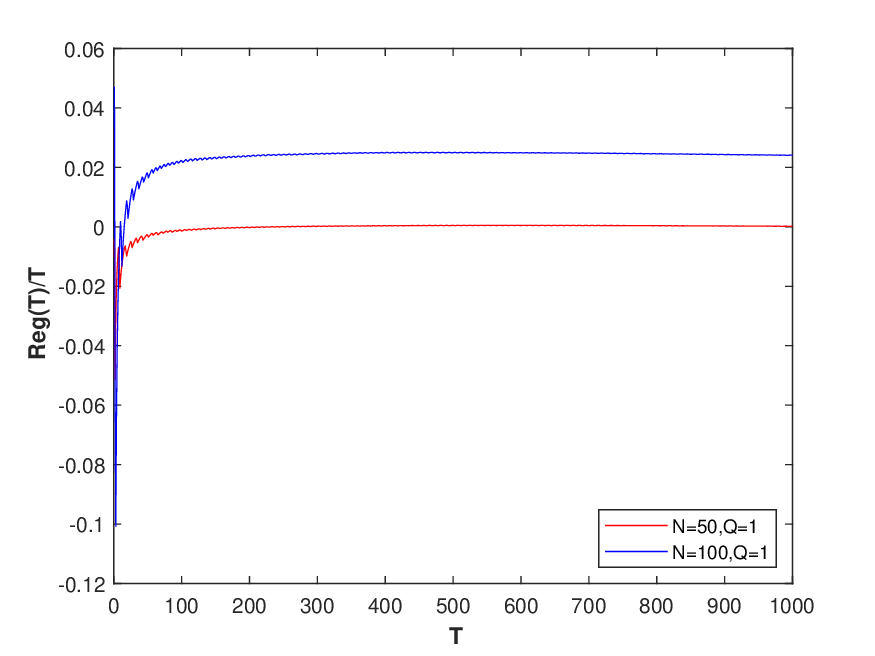}
\caption{Evolutions of $Reg(T)/T$ with $N=50$ and $N=100$ for $Q=1$.}
\label{f5}
\end{figure}

\begin{figure}[H]
\centering
\includegraphics[width=2.9in]{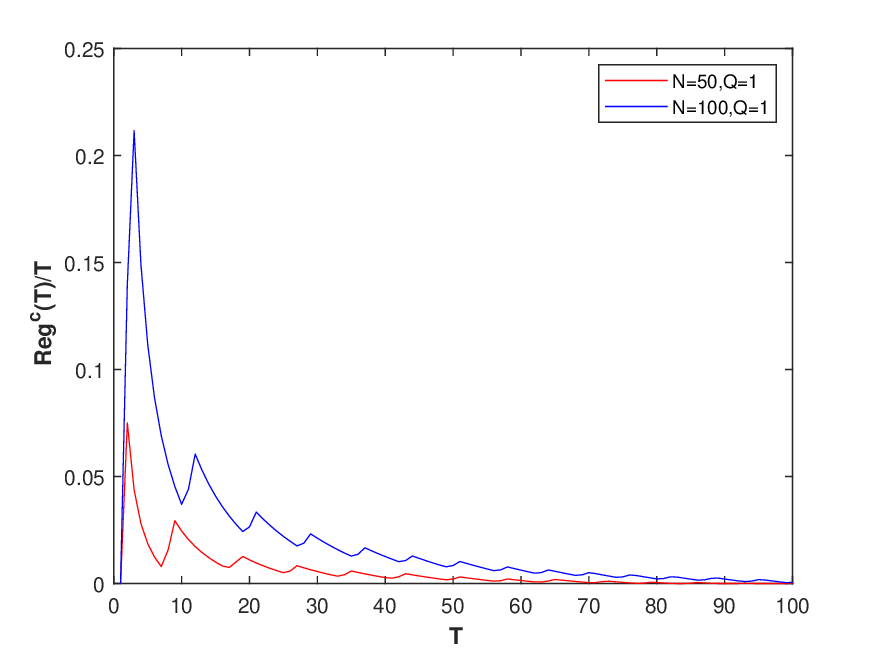}
\caption{Evolutions of $Reg^c(T)/T$ with $N=50$ and $N=100$ for $Q=1$.}
\label{f6}
\end{figure}

\begin{figure}[H]
\centering
\includegraphics[width=2.9in]{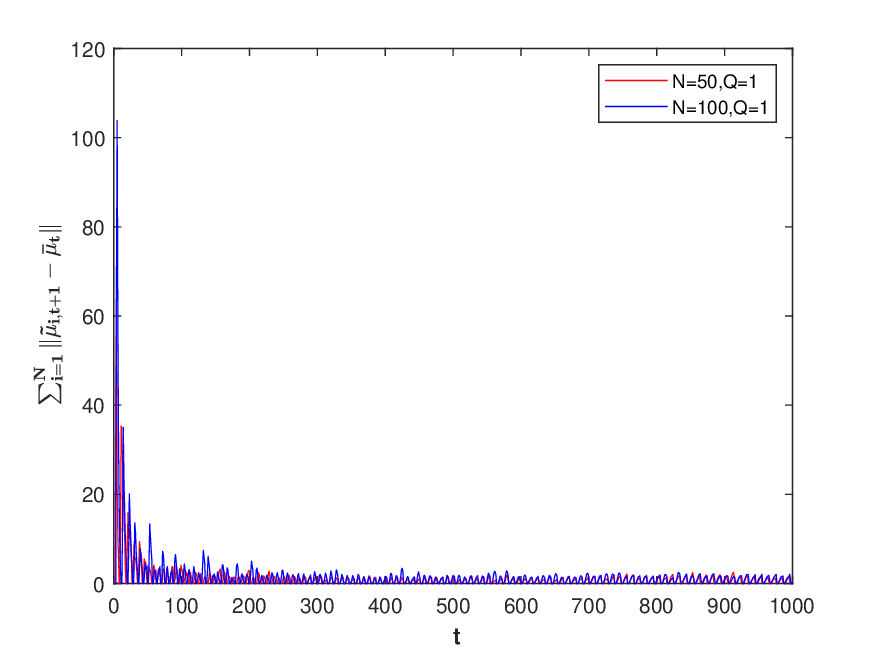}
\caption{Evolution of $\sum_{i=1}^N\|\tilde{\mu}_{i,t+1}-\bar{\mu}_t\|$ over $N$ agents.}
\label{f7}
\end{figure}

\section*{Appendix}

\subsection{Useful Lemmas}\label{s4.1}

\begin{lemma}\label{l0}
For any vector $z,v_1,\ldots,v_d\in\mathbb{R}^n$, there holds
\begin{align}
\|z\|&\leq\|z\|_1\leq\sqrt{n}\|z\|,         \nonumber\\
\|v_1+\cdots+v_d\|^2&\leq d(\|v_1\|^2+\cdots+\|v_d\|^2).      \nonumber
\end{align}
\end{lemma}

\begin{proof}
The first one is easy to be proved and can be also found in a great deal of literature. Thus, its proof is omitted here. To show the second one, consider first the case with $d=2$. Then one has that
\begin{align}
\|v_1+v_2\|^2&=\|v_1\|^2+\|v_2\|^2+2v_1^\top v_2          \nonumber\\
&\leq \|v_1\|^2+\|v_2\|^2+2\|v_1\| \|v_2\|          \nonumber\\
&\leq 2(\|v_1\|^2+\|v_2\|^2),                      \nonumber
\end{align}
where the Cauchy-Schwarz inequality has been used in the first inequality and the fact $2ab\leq a^2+b^2$ for any $a,b\in\mathbb{R}$ has been employed in the second inequality. By recursively using the same argument for the general case $d>2$, one can obtain the second asserted inequality in this lemma.
\end{proof}

A result on perturbed push-sum algorithms is listed below, which is cited from \cite{nedic2015distributed}.

\begin{lemma}\label{l1}
Consider the sequences $\{w_{i,t}\}$ with $w_{i,t}\in\mathbb{R}$ and $\{z_{i,t}\}$ with $z_{i,t}\in\mathbb{R}^{p}$, having the following dynamics:
\begin{align}
z_{i,t+1}&=\sum_{j=1}^N a_{ij,t}z_{j,t}+\epsilon_{i,t+1},                      \nonumber\\
w_{i,t+1}&=\sum_{j=1}^N a_{ij,t}w_{j,t},                                        \nonumber\\
\tilde{z}_{i,t+1}&=\frac{\sum_{j=1}^N a_{ij,t}z_{j,t}}{w_{i,t+1}},~~~\text{for}~i\in[N],t\geq 0         \label{40}
\end{align}
where $\epsilon_{i,t}$ is a perturbation for agent $i$ at time slot $t$. Denote by $\bar{z}_{t}=\frac{1}{N}\sum_{i=1}^N z_{i,t}$ the averaged variable of $z_{i,t}$'s. If Assumption \ref{a1} holds, then the following statement is true:
\begin{align}
\|\tilde{z}_{i,t+1}-\bar{z}_t\|\leq \frac{8}{r}\big(\lambda^t \|z_0\|_1+\sum_{k=1}^t \lambda^{t-k}\|\epsilon_k\|_1\big),        \nonumber
\end{align}
where $z_0:=col(z_{1,0},\ldots,z_{N,0})$, $\epsilon_k:=col(\epsilon_{1,k},\ldots,\epsilon_{N,k})$, $r:=\inf_{t=0,1,\ldots}(\min_{i\in[N]}\{A_t\cdots A_0\textbf{1}_N\}_i)$, and $\lambda\in (0,1)$, satisfying
\begin{align}
r\geq \frac{1}{N^{NQ}},~~~\lambda\leq \Big(1-\frac{1}{N^{NQ}}\Big)^{\frac{1}{NQ}}.     \nonumber
\end{align}
\end{lemma}

In the above lemma, the parameters $r,\lambda$ can be better selected when $A_t$ is doubly stochastic, i.e., balanced graphs, for all $t\geq 1$. Please refer to \cite{nedic2015distributed} for more details.

With Lemma \ref{l1} in place, it is straightforward to see that (\ref{ag4}) and (\ref{ag5}) can be rewritten in the perturbed form (\ref{40}) as
\begin{align}
\mu_{i,t+1}&=\hat{\mu}_{i,t}+\epsilon_{\mu_{i,t+1}},        \label{ag4'}\\
y_{i,t+1}&=\hat{y}_{i,t}+\epsilon_{y_{i,t+1}},                 \label{ag5'}
\end{align}
where
\begin{align}
\epsilon_{\mu_{i,t+1}}&:=\Big[\hat{\mu}_{i,t}+\alpha_t \Big(\frac{\hat{y}_{i,t}}{w_{i,t+1}}-\beta_t\hat{\mu}_{i,t}\Big)\Big]_+-\hat{\mu}_{i,t},   \label{41}\\
\epsilon_{y_{i,t+1}}&:=g_{i}(x_{i,t+1})-g_{i}(x_{i,t}).                   \label{42}
\end{align}

To move forward, for notational simplicity, let us denote
\begin{align}
\tilde{\mu}_{i,t+1}&=\frac{\hat{\mu}_{i,t}}{w_{i,t+1}},~~~\bar{\mu}_t=\frac{1}{N}\sum_{i=1}^N \mu_{i,t},             \nonumber\\
\tilde{y}_{i,t+1}&=\frac{\hat{y}_{i,t}}{w_{i,t+1}},~~~\bar{y}_t=\frac{1}{N}\sum_{i=1}^N y_{i,t},~\forall~i\in[N], t\geq 0.              \label{43}
\end{align}

For the purpose of facilitating the following analysis, it is helpful to present some preliminary results below.

\begin{lemma}\label{l2}
If Assumption \ref{a1} holds, then
\begin{align}
r\leq w_{i,t}\leq N,~~r\leq 1,~~\|\bar{y}_t\|\leq B_g,~~\forall i\in[N],t\geq 0,          \label{le1}
\end{align}
where $r$ is defined in Lemma \ref{l1} and $B_g$ is given in (\ref{14}).
\end{lemma}

\begin{proof}
The proof can be found in Appendix D.
\end{proof}

\begin{lemma}\label{l3}
Under Assumption \ref{a1}, there exists a constant $B_y>0$ such that for all $i\in[N]$ and $t\geq 1$
\begin{align}
\|y_{i,t}\|&\leq B_y,~~~~~~~~~~~\|\hat{y}_{i,t}\|\leq B_y,                           \label{44}\\
\|\hat{\mu}_{i,t}\|&\leq \frac{w_{i,t+1}B_y}{\beta_t r^2},~~~\|\mu_{i,t}\|\leq \frac{w_{i,t+1}B_y}{\beta_t r^2a},       \label{45}
\end{align}
where $a$ is given in Assumption \ref{a1}.
\end{lemma}

\begin{proof}
The proof can be found in Appendix E.
\end{proof}

\begin{lemma}\label{l4a}
Under Assumptions \ref{a1} and \ref{a2}, there holds
\begin{align}
&\sum_{t=1}^T\sum_{i=1}^N \|\tilde{\mu}_{i,t+1}-\bar{\mu}_t\|\leq \frac{8N^2(N+1)\sqrt{m}B_y}{r^3(1-\lambda)}\sum_{k=0}^{T-1}\alpha_{k},               \label{78}\\
&\sum_{t=1}^T\sum_{i=1}^N \|\tilde{y}_{i,t+1}-\bar{y}_t\|\leq \frac{8N\lambda\|y_0\|_1}{r(1-\lambda)}          \nonumber\\
&\hspace{0.3cm}+\frac{8\sqrt{m}N^2C_fC_g}{r(1-\lambda)}\sum_{k=0}^{T-1}\alpha_{k}+\frac{8\sqrt{m}N^2B_yC_g^2}{r^3(1-\lambda)}\sum_{k=0}^{T-1}\frac{\alpha_k}{\beta_k}.     \label{83}
\end{align}
\end{lemma}
\begin{proof}
The proof can be found in Appendix F.
\end{proof}

Equipped with the above results, it is now ready to present the results on the disagreement of $L_t(x,\mu)$ at different points.

\begin{lemma}\label{l4}
Let $x_t=col(x_{1,t},\ldots,x_{N,t})$. Under Assumptions \ref{l1} and \ref{l2}, then for all $x=col(x_1,\ldots,x_N)\in X$ and $\mu\in\mathbb{R}_+^m$,
\begin{align}
&L_t(x_t,\bar{\mu}_t)-L_t(x,\bar{\mu}_t)            \nonumber\\
&\hspace{0.1cm}\leq \frac{1}{2\alpha_t}\sum_{i=1}^N \big(\|x_{i,t}-x_i\|^2-\|x_{i,t+1}-x_i\|^2\big)         \nonumber\\
&\hspace{0.4cm}+\frac{\alpha_t N}{2}\Big(C_f+\frac{C_g B_y}{\beta_t r^2}\Big)^2+2B_g\sum_{i=1}^N \|\tilde{\mu}_{i,t+1}-\bar{\mu}_t\|,                    \label{56}\\
&L_t(x_t,\mu)-L_t(x_t,\bar{\mu}_t)               \nonumber\\
&\hspace{0.1cm}\leq \frac{N}{2\alpha_t} \big(\|\bar{\mu}_{t}-\mu\|^2-\|\bar{\mu}_{t+1}-\mu\|^2\big)         \nonumber\\
&\hspace{0.4cm}+\Big(\|\mu\|+\frac{B_y}{\beta_t r^2}\Big)\sum_{i=1}^N \|\tilde{y}_{i,t+1}-\bar{y}_t\|+\frac{2\alpha_t N^3 B_y^2(r+2)}{r^5}       \nonumber\\
&\hspace{0.4cm}+\Big(B_g+\frac{2NB_y}{r^2}\Big)\sum_{i=1}^N \|\tilde{\mu}_{i,t+1}-\bar{\mu}_t\|+\frac{N^2\beta_t}{2}\|\mu\|^2.            \label{57}
\end{align}
\end{lemma}

\begin{proof}
The proof can be found in Appendix G.
\end{proof}

\subsection{Proof of Theorem \ref{t1}}\label{s4.2}

The proof is divided into three parts.

{\bf Part 1: To show (\ref{29}).} By virtue of Lemma \ref{l4}, it can be obtained that for all $x\in X$ and $\mu\in\mathbb{R}^m_+$
\begin{align}
&L_t(x_t,\mu)-L_t(x,\bar{\mu}_t)           \nonumber\\
&=L_t(x_t,\mu)-L_t(x_t,\bar{\mu}_t)+L_t(x_t,\bar{\mu}_t)-L_t(x,\bar{\mu}_t)      \nonumber\\
&\leq \frac{1}{2\alpha_t}\sum_{i=1}^N (\|x_{i,t}-x_i\|^2-\|x_{i,t+1}-x_i\|^2)      \nonumber\\
&\hspace{0.4cm} +\frac{N}{2\alpha_t} (\|\bar{\mu}_{t}-\mu\|^2-\|\bar{\mu}_{t+1}-\mu\|^2)      \nonumber\\
&\hspace{0.4cm} +\Big(3B_g+\frac{2NB_y}{r^2}\Big)\sum_{i=1}^N \|\tilde{\mu}_{i,t+1}-\bar{\mu}_t\|            \nonumber\\
&\hspace{0.4cm} +\Big(\|\mu\|+\frac{B_y}{\beta_t r^2}\Big)\sum_{i=1}^N \|\tilde{y}_{i,t+1}-\bar{y}_t\|+\frac{N^2\beta_t}{2}\|\mu\|^2          \nonumber\\
&\hspace{0.4cm} +\frac{2\alpha_t N^3B_y^2(r+2)}{r^5}+\frac{N\alpha_t}{2}\Big(C_f+\frac{C_gB_y}{\beta_t r^2}\Big)^2.              \label{68}
\end{align}
Meanwhile, by letting $x=x_t^*$ with $x^*=col(x_{1,t}^*,\ldots,x_{N,t}^*)$ being given in (\ref{10}), it is easy to verify that
\begin{align}
&L_t(x_t,\mu)-L_t(x_t^*,\bar{\mu}_t)-\frac{N^2\beta_t}{2}\|\mu\|^2           \nonumber\\
&=\sum_{i=1}^N f_{i,t}(x_{i,t})+\mu^\top \sum_{i=1}^N g_{i}(x_{i,t})        \nonumber\\
&\hspace{0.4cm}-\sum_{i=1}^N f_{i,t}(x_{i,t}^*)-\bar{\mu}_t^\top\sum_{i=1}^N g_{i}(x_{i,t}^*)-\frac{N^2\beta_t}{2}\|\mu\|^2      \nonumber\\
&\geq \sum_{i=1}^N f_{i,t}(x_{i,t})-\sum_{i=1}^N f_{i,t}(x_{i,t}^*)             \nonumber\\
&\hspace{0.4cm}+\mu^\top \sum_{i=1}^N g_{i}(x_{i,t})-\frac{N^2\beta_t}{2}\|\mu\|^2,              \label{69}
\end{align}
where the inequality is obtained by resorting to $\bar{\mu}_t\geq \textbf{0}$ and $\sum_{i=1}^N g_{i}(x_{i,t}^*)\leq \textbf{0}$. For ease of exposition, define
\begin{align}
g_e(\mu):=\mu^\top \sum_{t=1}^T\sum_{i=1}^N g_{i}(x_{i,t})-\frac{N^2\|\mu\|^2}{2}\sum_{t=1}^T\beta_t.             \label{70}
\end{align}
By selecting $x=x_t^*$ and using the fact $(a+b)^2\leq 2(a^2+b^2)$ for any $a,b\in\mathbb{R}$ for the last term in (\ref{68}), combining (\ref{68}) with (\ref{69}) and summing over $t\in[T]$ yield that for all $\mu\in\mathbb{R}_+^m$,
\begin{align}
&\sum_{t=1}^T\sum_{i=1}^N f_{i,t}(x_{i,t})-\sum_{t=1}^T\sum_{i=1}^N f_{i,t}(x_{i,t}^*)+g_e(\mu)        \nonumber\\
&\leq \underbrace{\sum_{t=1}^T\frac{1}{2\alpha_t}\sum_{i=1}^N(\|x_{i,t+1}-x_{i,t+1}^*\|^2-\|x_{i,t+1}-x_{i,t}^*\|^2)}_{=:S_0}        \nonumber\\
&\hspace{0.4cm}+\underbrace{\sum_{t=1}^T\frac{1}{2\alpha_t}\sum_{i=1}^N (\|x_{i,t}-x_{i,t}^*\|^2-\|x_{i,t+1}-x_{i,t+1}^*\|^2)}_{=:S_1}      \nonumber\\
&\hspace{0.4cm} +\underbrace{\sum_{t=1}^T\frac{N}{2\alpha_t} (\|\bar{\mu}_{t}-\mu\|^2-\|\bar{\mu}_{t+1}-\mu\|^2)}_{=:S_2(\mu)}      \nonumber\\
&\hspace{0.4cm} +\underbrace{\Big(3B_g+\frac{2NB_y}{r^2}\Big)\sum_{t=1}^T\sum_{i=1}^N \|\tilde{\mu}_{i,t+1}-\bar{\mu}_t\|}_{=:S_3}            \nonumber\\
&\hspace{0.4cm} +\underbrace{\sum_{t=1}^T\Big(\|\mu\|+\frac{B_y}{\beta_t r^2}\Big)\sum_{i=1}^N \|\tilde{y}_{i,t+1}-\bar{y}_t\|}_{=:S_4(\mu)}          \nonumber\\
&\hspace{0.4cm} +\underbrace{N\Big(C_f^2+\frac{2N^2B_y^2(r+2)}{r^5}\Big)\sum_{t=1}^T \alpha_t}_{=:S_5}+\underbrace{\frac{NB_y^2C_g^2}{r^4}\sum_{t=1}^T\frac{\alpha_t}{\beta_t^2}}_{=:S_6}.             \label{71}
\end{align}

In the following the terms $S_i,i=0,1,\ldots,6$ are gradually analyzed. First, it is easy to obtain that
\begin{align}
S_0&=\sum_{t=1}^T\frac{1}{2\alpha_t}\sum_{i=1}^N(x_{i,t+1}^*-x_{i,t}^*)^\top(x_{i,t}^*+x_{i,t+1}^*-2x_{i,t+1})    \nonumber\\
&\leq 2B_x\sum_{t=1}^T\frac{1}{\alpha_t}\sum_{i=1}^N\|x_{i,t+1}^*-x_{i,t}^*\|            \nonumber\\
&=2B_x V_T,         \label{rp5}
\end{align}
where the inequality was resulted from the Cauchy-Schwarz inequality and (\ref{12}), and $V_T$ is defined in (\ref{VT}).

For $S_1$, some calculations can lead to that
\begin{align}
S_1&=\frac{1}{2\alpha_1}\sum_{i=1}^N\|x_{i,1}-x_{i,1}^*\|^2-\frac{1}{2\alpha_T}\sum_{i=1}^N\|x_{i,T+1}-x_{i,T+1}^*\|^2       \nonumber\\
&\hspace{0.4cm}+\frac{1}{2}\sum_{t=2}^T\Big(\frac{1}{\alpha_t}-\frac{1}{\alpha_{t-1}}\Big)\sum_{i=1}^N\|x_{i,t}-x_{i,t}^*\|^2       \nonumber\\
&\leq \frac{1}{\alpha_1}\sum_{i=1}^N(\|x_{i,1}\|^2+\|x_{i,1}^*\|^2)       \nonumber\\
&\hspace{0.4cm}+\sum_{t=2}^T\Big(\frac{1}{\alpha_t}-\frac{1}{\alpha_{t-1}}\Big)\sum_{i=1}^N(\|x_{i,t}\|^2+\|x_{i,t}^*\|^2)       \nonumber\\
&\leq \frac{2NB_x^2}{\alpha_T},               \label{72}
\end{align}
where Lemma \ref{l0} and $\frac{1}{\alpha_t}-\frac{1}{\alpha_{t-1}}>0$ have been leveraged for obtaining the first inequality, and $\|x_{i,t}^*\|\leq B_x,\|x_{i,t}\|\leq B_x,\forall i\in[N],t\geq 0$ have been used for inferring the last inequality.

Similarly, by letting $\mu=\textbf{0}$, one can have that
\begin{align}
S_2(\textbf{0})&=\frac{N}{2\alpha_1}\|\bar{\mu}_{1}\|^2-\frac{N}{2\alpha_T}\|\bar{\mu}_{T+1}\|^2       \nonumber\\
&\hspace{0.4cm}+\frac{N}{2}\sum_{t=2}^T(\frac{1}{\alpha_t}-\frac{1}{\alpha_{t-1}})\|\bar{\mu}_{t}\|^2          \nonumber\\
&\leq \frac{N^3B_y^2}{2a^2r^4\alpha_T\beta_T^2},               \label{73}
\end{align}
where we have made use of $\|\bar{\mu}_{t}\|\leq \sum_{i=1}^N\|\mu_{i,t}\|/N\leq w_{i,t+1}B_{y}/(a r^2\beta_t)\leq NB_y/(ar^2\beta_T)$ by Lemma \ref{l3} and $w_{i,t+1}\leq N$ in Lemma \ref{l2} for obtaining the inequality.

To bound $S_3$, invoking (\ref{78}) yields that
\begin{align}
S_3\leq \Big(3B_g+\frac{2NB_y}{r^2}\Big)\frac{8N^2(N+1)\sqrt{m}B_y}{r^3(1-\lambda)}\sum_{k=0}^{T-1}\alpha_{k}.      \label{79}
\end{align}

To bound $S_4(\mu)$ for $\mu=\textbf{0}$, by using (\ref{83}) and observing that $\beta_t\geq\beta_T$ for $t\leq T$, one can obtain that
\begin{align}
S_{4}(\textbf{0})&\leq \frac{B_y}{r^2\beta_T}\sum_{t=1}^T \sum_{i=1}^N \|\tilde{y}_{i,t+1}-\bar{y}_t\|           \nonumber\\
&\leq \frac{B_y}{r^2}\bigg[\frac{8N\lambda\|y_0\|_1}{r(1-\lambda)\beta_T}+\frac{8\sqrt{m}N^2C_fC_g}{r(1-\lambda)}\sum_{k=0}^{T-1}\frac{\alpha_{k}}{\beta_T}          \nonumber\\
&\hspace{1.3cm}+\frac{8\sqrt{m}N^2B_yC_g^2}{r^3(1-\lambda)}\sum_{k=0}^{T-1}\frac{\alpha_k}{\beta_T^2}\bigg].               \label{84}
\end{align}

With regard to $S_6$, it can be concluded that
\begin{align}
S_6&=\frac{N B_y^2C_g^2}{r^4}\sum_{t=1}^T t^{2\kappa-\frac{1}{2}}                \nonumber\\
&\leq \frac{N B_y^2C_g^2}{r^4}\Big(1+\int_1^T t^{2\kappa-\frac{1}{2}}dt\Big)            \nonumber\\
&=\frac{N B_y^2C_g^2}{r^4}\Big(1-\frac{2}{1+4\kappa}+\frac{2T^{\frac{1}{2}+2\kappa}}{1+4\kappa}\Big).       \label{85}
\end{align}

Note that $\sum_{k=0}^{T-1}\alpha_k\leq 2+\int_1^{T-1}t^{-1/2}dt=2(T-1)^{1/2}=O(T^{1/2})$ and also $\sum_{k=1}^{T}\alpha_k=O(T^{1/2})$. Thus, it is easy to verify that
\begin{align}
S_0=O_+(V_T),~~~&S_1=O_+(T^{\frac{1}{2}}),~~~S_2(\textbf{0})=O_+(T^{\frac{1}{2}+2\kappa}),     \nonumber\\
&\hspace{-0.8cm}S_3=O_+(T^{\frac{1}{2}}),~~~S_4(\textbf{0})=O_+(T^{\frac{1}{2}+2\kappa}),     \nonumber\\
&\hspace{-0.8cm}S_5=O_+(T^{\frac{1}{2}}),~~~S_6=O_+(T^{\frac{1}{2}+2\kappa}),                 \nonumber
\end{align}
which together with (\ref{71}) and $g_e(\textbf{0})=0$ completes the proof of (\ref{29}) in Theorem \ref{t1}.

{\bf Part 2: To show (\ref{30}).} Note that (\ref{71}) still holds for all $\mu\in\mathbb{R}_+^m$, when $x_{i,t}^*$ and $x_{i,t+1}^*$ are replaced with $x_i$, where $x=col(x_1,\ldots,x_N)$ is any point in $\mathcal{X}$ defined in (\ref{8}). In this case, $S_0=0$. Meanwhile, it is straightforward to verify that function $g_e(\mu)$, defined in (\ref{70}), can achieve its maximal value
\begin{align}
\frac{1}{2N^2\sum_{t=1}^T \beta_t}\Big\|\Big[\sum_{t=1}^T \sum_{i=1}^N g_{i}(x_{i,t})\Big]_+\Big\|^2          \label{86}
\end{align}
when $\mu=\mu_0$, where
\begin{align}
\mu_0:=\frac{\Big[\sum_{t=1}^T \sum_{i=1}^N g_{i}(x_{i,t})\Big]_+}{N^2\sum_{t=1}^T \beta_t},              \label{87}
\end{align}
which together with (\ref{71}) results in
\begin{align}
&\sum_{t=1}^T\sum_{i=1}^N f_{i,t}(x_{i,t})-\sum_{t=1}^T\sum_{i=1}^N f_{i,t}(x_{i})+\frac{(Reg^c(T))^2}{2N^2\sum_{t=1}^T \beta_t}        \nonumber\\
&\leq \sum_{t=1}^T\frac{1}{2\alpha_t}\sum_{i=1}^N (\|x_{i,t}-x_{i}\|^2-\|x_{i,t+1}-x_{i}\|^2)           \nonumber\\
&\hspace{0.4cm}+S_2(\mu_0)+S_3+S_4(\mu_0)+S_5+S_6.             \label{88}
\end{align}
Simple manipulations lead to that for $T\geq 4$ and $\kappa\in (0,1/4)$
\begin{align}
&\sum_{t=1}^T \beta_t\geq \int_1^T t^{-\kappa}dt =\frac{T^{1-\kappa}-1}{1-\kappa}\geq\frac{T^{1-\kappa}}{2(1-\kappa)},        \label{89}\\
&\sum_{t=1}^T \beta_t\leq 1+\int_1^T t^{-\kappa}dt =\frac{T^{1-\kappa}-\kappa}{1-\kappa}\leq\frac{T^{1-\kappa}}{1-\kappa},      \label{90}
\end{align}
which, together with (\ref{14}), gives rise to
\begin{align}
\|\mu_0\|\leq \frac{TB_g}{N\sum_{t=1}^T\beta_t}\leq \frac{2B_g(1-\kappa)T^{\kappa}}{N}.              \label{91}
\end{align}

By resorting to the similar arguments to bound $S_i$'s in (\ref{71}) and further applying (\ref{89})-(\ref{91}), we can bound the right-hand terms of (\ref{88}) as
\begin{align}
&\sum_{t=1}^T\sum_{i=1}^N f_{i,t}(x_{i,t})-\sum_{t=1}^T\sum_{i=1}^N f_{i,t}(x_{i})+\frac{(Reg^c(T))^2}{2N^2\sum_{t=1}^T \beta_t}        \nonumber\\
&=O_+(T^{\frac{1}{2}+2\kappa}).             \label{92}
\end{align}

Additionally, with reference to (\ref{12}) and (\ref{15}), one has that
\begin{align}
&\sum_{t=1}^T\sum_{i=1}^N f_{i,t}(x_{i,t})-\sum_{t=1}^T\sum_{i=1}^N f_{i,t}(x_{i})       \nonumber\\
&=\sum_{t=1}^T\sum_{i=1}^N (f_{i,t}(x_{i,t})-f_{i,t}(x_{i}))          \nonumber\\
&\geq -\sum_{t=1}^T\sum_{i=1}^N C_f\|x_{i,t}-x_{i}\|                   \nonumber\\
&\geq -2NTC_fB_x.                                        \label{93}
\end{align}
Inserting (\ref{93}) to (\ref{92}) gives that
\begin{align}
(Reg^c(T))^2&\leq \sum_{t=1}^T \beta_t\cdot O_+(T^{\frac{1}{2}+2\kappa})+4N^3C_fB_xT\sum_{t=1}^T\beta_t                            \nonumber\\
&= O_+(T^{\frac{3}{2}+\kappa})+O_+(T^{2-\kappa})             \nonumber\\
&= O_+(T^{2-\kappa}),                                      \label{94}
\end{align}
where we have employed (\ref{90}) to obtain the first equality, and $3/2+\kappa<2-\kappa$ due to $\kappa<1/4$ for the second equality. Obviously, (\ref{94}) is equivalent to (\ref{30}).

{\bf Part 3: To show (\ref{29b}).} In light of Assumption \ref{a2}.3, the Lagrangian function $L_t(x,\mu)$ in (\ref{18}) indeed has saddle points for all $t\in[T]$. Denote by $\mu_t^*$ an optimal dual variable of $L_{t}(x,\mu)$ corresponding to $x_t^*$. It is known by Lemma 1 in \cite{nedic2009approximate} that $\mu_t^*$ is bounded and the upper bound is independent of $t$ due to (\ref{13})-(\ref{14}). Therefore, one has that $L_t(x_t^*,\mu_t^*)\leq L_t(x,\mu_t^*)$ for all $x\in X$, by which choosing $x=x_t=col(x_{1,t},\ldots,x_{N,t})$ further implies that
\begin{align}
f_t(x_t^*)+(\mu_t^*)^\top g(x_t^*)\leq f_t(x_t)+(\mu_t^*)^\top g(x_t).          \label{rpp3}
\end{align}

Note that $(\mu_t^*)^\top g(x_t^*)=0$ by the optimality criteria. It can be then obtained that
\begin{align}
\sum_{i=1}^N f_{i,t}(x_{i,t})-\sum_{i=1}^N f_{i,t}(x_{i,t}^*)&\geq -(\mu_t^*)^\top\sum_{i=1}^N g_i(x_{i,t})                  \nonumber\\
&\hspace{-1.1cm}\geq -(\textbf{1}_m\otimes\bar{\mu}^*)^\top\sum_{i=1}^N g_i(x_{i,t}),          \label{rp8}
\end{align}
where the second inequality has employed $\sum_{i=1}^N g_i(x_{i,t})>\textbf{0}$ for the worst case studied in Theorem \ref{t1}, and
\begin{align}
\bar{\mu}^*:=\max_{t\in[T],l\in[m]}\big\{\mu_{t}^*\big\}_l         \label{mu}
\end{align}
with $\{\cdot\}_l$ denoting the $l$-th component of a vector.

By summing (\ref{rp8}) over $t\in[T]$ and invoking Lemma \ref{l0}, it is straightforward to obtain that
\begin{align}
&\sum_{t=1}^T\sum_{i=1}^N f_{i,t}(x_{i,t})-\sum_{t=1}^T\sum_{i=1}^N f_{i,t}(x_{i,t}^*)           \nonumber\\
&\geq -(\textbf{1}_m\otimes\bar{\mu}^*)^\top\sum_{t=1}^T\sum_{i=1}^N g_i(x_{i,t})      \nonumber\\
&= -(\textbf{1}_m\otimes\bar{\mu}^*)^\top\Big[\sum_{t=1}^T\sum_{i=1}^N g_i(x_{i,t})\Big]_+      \nonumber\\
&=-\bar{\mu}^*\Big\|\Big[\sum_{t=1}^T\sum_{i=1}^N g_i(x_{i,t})\Big]_+\Big\|_1        \nonumber\\
&\geq -\sqrt{m}\bar{\mu}^*\Big\|\Big[\sum_{t=1}^T\sum_{i=1}^N g_i(x_{i,t})\Big]_+\Big\|     \nonumber\\
&=-\sqrt{m}\bar{\mu}^* Reg^c(T),          \label{rpp4}
\end{align}
which, in conjunction with (\ref{30}), leads to (\ref{29b}). This completes the proof of Theorem \ref{t1}.

\subsection{Proof of Theorem \ref{t2}}\label{s4.3}

This section gives the proof of Theorem \ref{t2} when $f_{i,t}$'s are independent of time for all $i\in[N]$, denoted by $f_{i}$ in this section. Note that $V_T=0$ in this case.

Let us first prove (\ref{35}). By using the same argument as Theorem \ref{t1}, it can be obtained that
\begin{align}
Reg(T)=O_+(T^{\frac{1}{2}+2\kappa}).     \label{rp17}
\end{align}

Appealing to the convexity of $f_i$'s can lead to
\begin{align}
\frac{Reg(T)}{T}&=\sum_{t=1}^T\frac{1}{T}\sum_{i=1}^{N}f_i(x_{i,t})-\sum_{i=1}^{N}f_i(x_{i}^*)          \nonumber\\
&\geq \sum_{i=1}^{N}f_i(\bar{x}_{i,T})-\sum_{i=1}^{N}f_i(x_{i}^*),           \label{rp18}
\end{align}
which, together with (\ref{rp17}), gives rise to the assertion (\ref{35}).

To show (\ref{36}), define
\begin{align}
f(x)&=\sum_{i=1}^N f_i(x_i),               \label{96}\\
L(x,\mu)&=f(x)+\mu^\top g(x),             \label{97}
\end{align}
where $g(x)=\sum_{i=1}^N g_i(x_i)$. Note that in light of Assumption \ref{a2}.3, the Lagrangian function $L(x,\mu)$ indeed has saddle points. Now, invoking the property of saddle points can imply that $L(x^*,\mu^*)\leq L(x,\mu^*)$ for all $x\in X$, where $\mu^*\in\mathbb{R}_+^m$ is an optimal dual variable, which is equivalent to
\begin{align}
f(x^*)+(\mu^*)^\top g(x^*)\leq f(x_t)+(\mu^*)^\top g(x_t)             \label{98}
\end{align}
when letting $x=x_t:=col(x_{1,t},\ldots,x_{N,t})$. Then, summing (\ref{98}) over $t$ gives rise to
\begin{align}
&\sum_{t=1}^T\sum_{i=1}^N f_i(x_{i,t})-\sum_{t=1}^T\sum_{i=1}^N f_i(x_i^*)      \nonumber\\
&\geq -(\mu^*)^\top \sum_{t=1}^T\sum_{i=1}^N g_i(x_{i,t})        \nonumber\\
&\geq -(\mu^*)^\top \Big[\sum_{t=1}^T\sum_{i=1}^N g_i(x_{i,t})\Big]_+,             \label{99}
\end{align}
where we have employed the fact that $(\mu^*)^\top g(x^*)=0$ in the first inequality, and $\sum_{t=1}^T\sum_{i=1}^N g_i(x_{i,t})\leq \big[\sum_{t=1}^T\sum_{i=1}^N g_i(x_{i,t})\big]_+$ and $\mu^*\in\mathbb{R}_+^m$ in the last inequality.

Inserting (\ref{99}) into (\ref{92}) yields that
\begin{align}
\frac{(Reg^c(T))^2}{2N^2\sum_{t=1}^T \beta_t}-(\mu^*)^\top \Big[\sum_{t=1}^T\sum_{i=1}^N g_i(x_{i,t})\Big]_+=O_+(T^{\frac{1}{2}+2\kappa}),             \label{100}
\end{align}
which implies that
\begin{align}
&\Big\|\Big[\sum_{t=1}^T\sum_{i=1}^N g_i(x_{i,t})\Big]_+ -N^2\mu^*\sum_{t=1}^T\beta_t\Big\|^2      \nonumber\\
&=N^4\|\mu^*\|^2 \Big(\sum_{t=1}^T\beta_t\Big)^2 +O(T^{\frac{1}{2}+2\kappa})\cdot\sum_{t=1}^T \beta_t.             \label{101}
\end{align}
With reference to (\ref{90}), it can be obtained by (\ref{101}) that
\begin{align}
&\Big\|\Big[\sum_{t=1}^T\sum_{i=1}^N g_i(x_{i,t})\Big]_+ -N^2\mu^*\sum_{t=1}^T\beta_t\Big\|^2         \nonumber\\
&=\left\{
    \begin{array}{ll}
      O(T^{2-2\kappa}), & \kappa\in(0,\frac{1}{6}] \\
      O(T^{\frac{3}{2}+\kappa}), & \kappa\in[\frac{1}{6},\frac{1}{4}).
    \end{array}
  \right.             \label{102}
\end{align}
By considering the components, one has that for $l\in[m]$
\begin{align}
&\bigg|\Big\{\Big[\sum_{t=1}^T\sum_{i=1}^N g_i(x_{i,t})\Big]_+ -N^2\mu^*\sum_{t=1}^T\beta_t\Big\}_l\bigg|       \nonumber\\
&=\left\{
    \begin{array}{ll}
      O(T^{1-\kappa}), & \kappa\in(0,\frac{1}{6}] \\
      O(T^{\frac{3}{4}+\frac{\kappa}{2}}), & \kappa\in[\frac{1}{6},\frac{1}{4}).
    \end{array}
  \right.                \label{103}
\end{align}
Invoking the fact that $|a-b|\geq |a|-|b|$ for all $a,b\in\mathbb{R}$, it can be obtained that
\begin{align}
&\Big\{\Big[\sum_{t=1}^T\sum_{i=1}^N g_i(x_{i,t})\Big]_+\Big\}_l        \nonumber\\
&\leq \Big\{N^2\mu^*\sum_{t=1}^T\beta_t\Big\}_l+\left\{
    \begin{array}{ll}
      O(T^{1-\kappa}), & \kappa\in(0,\frac{1}{6}] \\
      O(T^{\frac{3}{4}+\frac{\kappa}{2}}), & \kappa\in[\frac{1}{6},\frac{1}{4})
    \end{array}
  \right.                \nonumber\\
&=\left\{
    \begin{array}{ll}
      O(T^{1-\kappa}), & \kappa\in(0,\frac{1}{6}] \\
      O(T^{\frac{3}{4}+\frac{\kappa}{2}}), & \kappa\in[\frac{1}{6},\frac{1}{4})
    \end{array}
  \right.                \label{104}
\end{align}
where (\ref{90}) has been used to obtain the equality.

Consequently, by Lemma \ref{l0} and (\ref{104}), one can obtain that
\begin{align}
Reg^c(T)&\leq \Big\|\Big[\sum_{t=1}^T\sum_{i=1}^N g_i(x_{i,t})\Big]_+\Big\|_1              \nonumber\\
&=\sum_{l=1}^m \Big\{\Big[\sum_{t=1}^T\sum_{i=1}^N g_{i}(x_{i,t})\Big]_+\Big\}_l           \nonumber\\
&=\left\{
    \begin{array}{ll}
      O(T^{1-\kappa}), & \kappa\in(0,\frac{1}{6}] \\
      O(T^{\frac{3}{4}+\frac{\kappa}{2}}), & \kappa\in[\frac{1}{6},\frac{1}{4}).
    \end{array}
  \right.                \label{105}
\end{align}

Now, appealing to the convexity of $g_i$'s, one has that
\begin{align}
\frac{Reg^c(T)}{T}&=\Big\|\Big[\sum_{t=1}^T\frac{1}{T}\sum_{i=1}^N g_i(x_{i,t})\Big]_+\Big\|          \nonumber\\
&\geq \Big\|\Big[\sum_{i=1}^N g_i(\bar{x}_{i,T})\Big]_+\Big\|,             \label{rp16}
\end{align}
which, in combination with (\ref{105}), proves the assertion (\ref{36}).

It remains to show (\ref{35b}). Invoking (\ref{99}) and (\ref{104}) yields that
\begin{align}
&\frac{1}{T}\sum_{t=1}^T\sum_{i=1}^N f_i(x_{i,t})-\sum_{i=1}^N f_i(x_i^*)      \nonumber\\
&\geq -\frac{(\mu^*)^\top}{T} \Big[\sum_{t=1}^T\sum_{i=1}^N g_i(x_{i,t})\Big]_+        \nonumber\\
&\geq -\frac{\mu_M^*}{T}\sum_{l=1}^m\Big\{\Big[\sum_{t=1}^T\sum_{i=1}^N g_i(x_{i,t})\Big]_+\Big\}_l     \nonumber\\
&=\left\{
    \begin{array}{ll}
      \Omega(-T^{-\kappa}), & \kappa\in(0,\frac{1}{6}] \\
      \Omega(-T^{-\frac{1}{4}+\frac{\kappa}{2}}), & \kappa\in[\frac{1}{6},\frac{1}{4}),
    \end{array}
  \right.             \label{nn1}
\end{align}
where $\mu_M^*:=\max_{l\in[m]}\{\mu^*\}_l$. This completes the proof.

\subsection{Proof of Lemma \ref{l2}}\label{ap2}

First, $w_{i,t}\geq r$ follows directly from the definition of $r$ in Lemma \ref{l1} once noting that $w_{i,0}=1$ for all $i\in[N]$. To prove $w_{i,t}\leq N$, it is easy to see that (\ref{ag1}) can be rewritten as
\begin{align}
w_{t+1}=A_tw_t,      \label{49}
\end{align}
where $w_t:=col(w_{1,t},\ldots,w_{N,t})$. By pre-multiplying $\textbf{1}^\top$ on both sides of (\ref{49}), one has that $\sum_{i=1}^N w_{i,t+1}=\sum_{i=1}^N w_{i,t}$ for all $t\geq 0$, which combines with the fact that $w_{i,0}=1$ for all $i\in[N]$ gives rise to that $\sum_{i=1}^N w_{i,t}=N$ for all $t\geq 0$. Observing the fact that $w_{i,t}\geq 0$, it can be concluded that $w_{i,t}\leq N$. Next, let us show that $r\leq 1$ by contradiction. If $r>1$, in view of $w_{i,t}\geq r$, then $\sum_{i=1}^N w_{i,t}\geq Nr>N$, contradicting $\sum_{i=1}^N w_{i,t}=N$. Hence, $r\leq 1$.

Finally, it remains to prove $\|\bar{y}_t\|\leq B_g$. In view of (\ref{ag5}), one can obtain that
\begin{align}
y_{t+1}=(A_t\otimes I_m)y_t +G(x_{t+1})-G(x_t),           \label{48}
\end{align}
where $x_t:=col(x_{1,t},\ldots,x_{N,t})$, $y_t:=col(y_{1,t},\ldots,y_{N,t})$, and $G(x_t):=col(g_{1}(x_{1,t}),\ldots,g_{N}(x_{N,t}))$. By pre-multiplying $\textbf{1}^\top$ on both sides of (\ref{48}), it can obtain that $\sum_{i=1}^N y_{i,t+1}=\sum_{i=1}^N y_{i,t}+g(x_{t+1})-g(x_t)$, and thus it yields that $\bar{y}_{t+1}-g(x_{t+1})/N=\bar{y}_t-g(x_t)/N$. Combining with $y_{i,0}=g_{i}(x_{i,0})$ results in that $\bar{y}_t=g(x_t)/N$ for all $t\geq 1$, thereby implying that $\|\bar{y}_t\|\leq B_g$ by (\ref{14}). This finishes the proof.

\subsection{Proof of Lemma \ref{l3}}\label{ap3}

Let us first prove (\ref{44}). In view of (\ref{ag5'}), it follows from Lemma \ref{l1} that
\begin{align}
\|\tilde{y}_{i,t+1}-\bar{y}_t\|\leq \frac{8}{r}\Big(\lambda^t\|y_0\|_1+\sum_{k=1}^t\lambda^{t-k}\|\epsilon_{y,k}\|_1\Big),       \label{47}
\end{align}
where $r,\lambda$ are given in Lemma \ref{l1}, $y_0:=col(y_{1,0},\ldots,y_{N,0})$ and $\epsilon_{y,k}:=col(\epsilon_{y_{1,k}},\ldots,\epsilon_{y_{N,k}})$. It is easy to see that $\|\epsilon_{y_{i,t+1}}\|_1\leq \sqrt{m}\|\epsilon_{y_{i,t+1}}\|\leq 2\sqrt{m}B_g$, where Lemma \ref{l0} and (\ref{14}) have been used to obtain the first and second inequalities, respectively. As a result, one has that $\sum_{k=1}^t \lambda^{t-k}\|\epsilon_{y,k}\|_1=\sum_{k=1}^t\sum_{i=1}^N \lambda^{t-k}\|\epsilon_{y_{i,k}}\|_1\leq 2N\sqrt{m}B_g/(1-\lambda)$, which together with (\ref{47}) implies that $\|\tilde{y}_{i,t+1}-\bar{y}_t\|$ is bounded. At this stage, the boundedness of $\|\tilde{y}_{i,t+1}-\bar{y}_t\|$ and $\bar{y}_t$ (by Lemma \ref{l2}) yields that $\tilde{y}_{i,t+1}$ is bounded, which together with the boundedness of $w_{i,t}$ in Lemma \ref{l2} leads to that $\hat{y}_{i,t}$ is bounded. At this point, invoking (\ref{ag5}), (\ref{14}) and boundedness of $\hat{y}_{i,t}$, it can be concluded that $y_{i,t}$ is bounded, that is, there exists $B_y>0$ such that $\|y_{i,t}\|\leq B_y$ and $\|\hat{y}_{i,t}\|\leq B_y$ for all $i\in[N]$, thus finishing the proof of (\ref{44}).

What follows is the proof of (\ref{45}). Let us first show that $\|\hat{\mu}_{i,t}\|\leq \frac{w_{i,t+1}B_y}{\beta_t r^2}$ by induction. It is easy to see that $\hat{\mu}_{i,0}\leq \frac{w_{i,1}B_y}{\beta_0 r^2}$ due to $\beta_0=1$ and $\mu_{i,0}=0$ for all $i\in[N]$. Assume now that it is true at time instant $t$ for all $i\in[N]$, and it suffices to show that it remains true at time $t+1$. At first step, it can be obtained that for all $i\in[N]$,
\begin{align}
\hat{\mu}_{i,t}+\alpha_t\Big(\frac{\hat{y}_{i,t}}{w_{i,t+1}}-\beta_t\hat{\mu}_{i,t}\Big)&=(1-\alpha_t\beta_t)\hat{\mu}_{i,t}+\frac{\alpha_t\hat{y}_{i,t}}{w_{i,t+1}}                   \nonumber\\
&\hspace{-0.2cm}\leq (1-\alpha_t\beta_t)\frac{w_{i,t+1}B_y}{\beta_t r^2}+\frac{\alpha_t B_y}{r}        \nonumber\\
&\hspace{-0.2cm}=\Big(1-\alpha_t\beta_t+\frac{r\alpha_t\beta_t}{w_{i,t+1}}\Big)\frac{w_{i,t+1}B_y}{\beta_t r^2}      \nonumber\\
&\hspace{-0.2cm}\leq \frac{w_{i,t+1}B_y}{\beta_t r^2},                                                            \label{50}
\end{align}
where we have used (\ref{44}) and $w_{i,t+1}\geq r$ to gain the first inequality and $w_{i,t+1}\geq r$ to obtain the last inequality. Therefore, in light of (\ref{ag4}) and (\ref{50}), one has that $\mu_{i,t+1}\leq w_{i,t+1}B_y/(\beta_tr^2)$ for all $i\in[N]$, thereby yielding that
\begin{align}
\hat{\mu}_{i,t+1}&=\sum_{j=1}^N a_{ij,t+1}\mu_{j,t+1}\leq \frac{B_y}{\beta_t r^2}\sum_{j=1}^N a_{ij,t+1} w_{i,t+1}         \nonumber\\
&= \frac{w_{i,t+2}B_y}{\beta_t r^2}                                \nonumber\\
&\leq \frac{w_{i,t+2}B_y}{\beta_{t+1} r^2},                          \label{52}
\end{align}
where (\ref{ag1}) and $\beta_{t+1}\leq \beta_t$ have been used to obtain the last equality and inequality, respectively. Therefore, the assertion $\hat{\mu}_{i,t}\leq w_{i,t+1}B_y/(\beta_{t}r^2)$ holds for all $t\geq 0$ and $i\in[N]$.

Now, note that $\hat{\mu}_{i,t}=\sum_{j=1}^N a_{ij,t}\mu_{j,t}\geq a_{ii,t}\mu_{i,t}$. It can be obtained that $\|\mu_{i,t}\|\leq \hat{\mu}_{i,t}/a_{ii,t}\leq w_{i,t+1}B_y/(\beta_t r^2a)$, where $a_{ii,t}\geq a$ in Assumption \ref{a1} has been used. This ends the proof.

\subsection{Proof of Lemma \ref{l4a}}

Invoking Lemma \ref{l1} implies that
\begin{align}
&\sum_{t=1}^T\sum_{i=1}^N \|\tilde{\mu}_{i,t+1}-\bar{\mu}_t\|         \nonumber\\
&\leq \frac{8N}{r}\sum_{t=1}^T\Big(\lambda^t\|\mu_0\|_1+\sum_{k=1}^t\lambda^{t-k}\|\epsilon_{\mu,k}\|_1\Big),               \label{74}
\end{align}
where $\epsilon_{\mu,k}:=col(\epsilon_{\mu_{1,k}},\ldots,\epsilon_{\mu_{N,k}})$ with $\epsilon_{\mu_{i,k}}$ being defined in (\ref{41}) and $\mu_{0}:=col(\mu_{1,0},\ldots,\mu_{N,0})$. In view of Lemma \ref{l0}, (\ref{41}), and (\ref{pr2}), we have that
\begin{align}
\|\epsilon_{\mu_{i,t+1}}\|_1&\leq \sqrt{m}\|\epsilon_{\mu_{i,t+1}}\|           \nonumber\\
&\leq \alpha_t\sqrt{m} \|\frac{\hat{y}_{i,t}}{w_{i,t+1}}-\beta_t\hat{\mu}_{i,t}\|       \nonumber\\
&\leq \frac{\alpha_t \sqrt{m}B_y(N+1)}{r^2},             \label{75}
\end{align}
where Lemmas \ref{l2} and \ref{l3} have been applied to obtain the last inequality. Therefore, in light of $\|\epsilon_{\mu,k}\|_1=\sum_{i=1}^N \|\epsilon_{\mu_{i,k}}\|_1$ and $\mu_{i,0}=0$ for all $i\in[N]$, it follows from (\ref{74}) and (\ref{75}) that
\begin{align}
&\sum_{t=1}^T\sum_{i=1}^N \|\tilde{\mu}_{i,t+1}-\bar{\mu}_t\|         \nonumber\\
&\hspace{1.1cm}\leq \frac{8 N^2(N+1)\sqrt{m}B_y}{r^3}\sum_{t=1}^T\sum_{k=1}^t\lambda^{t-k}\alpha_{k-1},               \label{76}
\end{align}
which, together with the fact that
\begin{align}
\sum_{t=1}^T\sum_{k=1}^t \lambda^{t-k}\alpha_{k-1}&=\sum_{t=0}^{T-1}\lambda^t\sum_{k=0}^{T-t-1}\alpha_k\leq \sum_{t=0}^{T-1}\lambda^t \sum_{k=0}^{T-1}\alpha_k           \nonumber\\
&\leq \frac{1}{1-\lambda}\sum_{k=0}^{T-1}\alpha_k,             \label{77}
\end{align}
results in (\ref{78}).

Similarly, to show (\ref{83}), one has by Lemma \ref{l1} that
\begin{align}
&\sum_{t=1}^T\sum_{i=1}^N \|\tilde{y}_{i,t+1}-\bar{y}_t\|         \nonumber\\
&\leq \frac{8N}{r}\sum_{t=1}^T\Big(\lambda^t\|y_0\|_1+\sum_{k=1}^t\lambda^{t-k}\|\epsilon_{y,k}\|_1\Big),               \label{80}
\end{align}
where $\epsilon_{y,k}:=col(\epsilon_{y_{1,k}},\ldots,\epsilon_{y_{N,k}})$ with $\epsilon_{y_{i,k}}$ being defined in (\ref{42}). In light of Lemma \ref{l0}, (\ref{42}), and (\ref{16}), one has that
\begin{align}
\|\epsilon_{y_{i,t+1}}\|_1&\leq \sqrt{m}\|\epsilon_{y_{i,t+1}}\|           \nonumber\\
&\leq \sqrt{m}C_g \|x_{i,t+1}-x_{i,t}\|.             \label{81}
\end{align}
Invoking (\ref{ag3}) and (\ref{pr2}) leads to that $\|x_{i,t+1}-x_{i,t}\|\leq \alpha_t\|s_{i,t+1}\|$, which together with (\ref{17a}), (\ref{17b}), (\ref{25}), and (\ref{45}) implies that
\begin{align}
\|\epsilon_{y_{i,t+1}}\|_1\leq \sqrt{m}C_g\alpha_t\Big(C_f+\frac{B_yC_g}{\beta_t r^2}\Big).             \label{82}
\end{align}

Then, similar to (\ref{75})-(\ref{77}), one can obtain (\ref{83}). This ends the proof.

\subsection{Proof of Lemma \ref{l4}}

To show (\ref{56}), invoking (\ref{ag3}) and (\ref{pr2}) yields that
\begin{align}
\|x_{i,t+1}-x_i\|^2&\leq \|x_{i,t}-x_i-\alpha_t s_{i,t+1}\|^2        \nonumber\\
&=\|x_{i,t}-x_i\|^2+\alpha_t^2\|s_{i,t+1}\|^2                \nonumber\\
&\hspace{0.35cm}-2\alpha_t s_{i,t+1}^\top (x_{i,t}-x_i),~~~\forall x\in X       \label{59}
\end{align}
in which, in view of (\ref{25}), the last term can be manipulated as
\begin{align}
&-2\alpha_t s_{i,t+1}^\top (x_{i,t}-x_i)          \nonumber\\
&= -2\alpha_t[\partial^\top f_{i,t}(x_{i,t})+ \tilde{\mu}_{i,t+1}^\top \partial^\top g_{i}(x_{i,t})](x_{i,t}-x_i)        \nonumber\\
&\leq -2\alpha_t [f_{i,t}(x_{i,t})-f_{i,t}(x_i)+\tilde{\mu}_{i,t+1}^\top(g_{i}(x_{i,t})-g_{i}(x_i))]                      \nonumber\\
&=-2\alpha_t[L_{i,t}(x_{i,t},\bar{\mu}_t)-L_{i,t}(x_{i},\bar{\mu}_t)            \nonumber\\
&\hspace{1.35cm}+(\tilde{\mu}_{i,t+1}-\bar{\mu}_t)^\top(g_{i}(x_{i,t})-g_{i}(x_i))],       \label{60}
\end{align}
where the convexity of $f_{i,t},g_{i}$ (i.e., (\ref{5})) and $\tilde{\mu}_{i,t+1}\geq \textbf{0}$ have been exploited for obtaining the inequality, and (\ref{19}) has been used in the last equality. Note that $\|s_{i,t+1}\|\leq C_f+C_gB_y/(\beta_t r^2)$ by (\ref{17a}), (\ref{17b}), (\ref{25}) and (\ref{45}). Consequently, by combining (\ref{59}) and (\ref{60}) with (\ref{14}), preforming summations over $i\in[N]$ leads to (\ref{56}), thus ending the proof of (\ref{56}).

It remains to show (\ref{57}). To do so, calculating the average of (\ref{ag4'}) over $i\in[N]$ leads to that $\bar{\mu}_{t+1}=\bar{\mu}_t+\frac{1}{N}\sum_{i=1}^N \epsilon_{\mu_{i,t+1}}$, by which invoking (\ref{pr2}) can yield that for all $\mu\in\mathbb{R}_+^m$,
\begin{align}
&\|\bar{\mu}_{t+1}-\mu\|^2          \nonumber\\
&=\Big\|\bar{\mu}_t-\mu+\frac{\sum_{i=1}^N\epsilon_{\mu_{i,t+1}}}{N}\Big\|^2         \nonumber\\
&\leq \|\bar{\mu}_{t}-\mu\|^2+\frac{1}{N}\sum_{i=1}^N\|\epsilon_{\mu_{i,t+1}}\|^2+\frac{2}{N}\sum_{i=1}^N\epsilon_{\mu_{i,t+1}}^\top(\bar{\mu}_t-\mu)              \nonumber\\
&\leq \|\bar{\mu}_{t}-\mu\|^2+\frac{\alpha_t^2}{N}\sum_{i=1}^N\Big\|\frac{\hat{y}_{i,t}}{w_{i,t+1}}-\beta_t\hat{\mu}_{i,t}\Big\|^2                      \nonumber\\
&\hspace{0.4cm}+\frac{2}{N}\sum_{i=1}^N\epsilon_{\mu_{i,t+1}}^\top(\bar{\mu}_t-\mu),                                                   \label{61}
\end{align}
where Lemma \ref{l0} has been applied to obtain the first inequality, and (\ref{pr2}) and (\ref{41}) have been utilized in the second inequality.

Let us now consider the term $\epsilon_{\mu_{i,t+1}}^\top(\bar{\mu}_t-\mu)$ in the last inequality of (\ref{61}). It can be obtained that
\begin{align}
&\epsilon_{\mu_{i,t+1}}^\top(\bar{\mu}_t-\mu)           \nonumber\\
&=\frac{1}{w_{i,t+1}}\epsilon_{\mu_{i,t+1}}^\top (w_{i,t+1}\bar{\mu}_t-w_{i,t+1}\mu)         \nonumber\\
&=\frac{1}{w_{i,t+1}}\epsilon_{\mu_{i,t+1}}^\top (\hat{\mu}_{i,t}-w_{i,t+1}\mu)+\epsilon_{\mu_{i,t+1}}^\top(\bar{\mu}_t-\tilde{\mu}_{i,t+1})        \nonumber\\
&=\Big(\epsilon_{\mu_{i,t+1}}-\alpha_t\Big(\frac{\hat{y}_{i,t}}{w_{i,t+1}}-\beta_t\hat{\mu}_{i,t}\Big)\Big)^\top \frac{(\hat{\mu}_{i,t}-w_{i,t+1}\mu)}{w_{i,t+1}}      \nonumber\\
&\hspace{0.1cm}+\alpha_t\Big(\frac{\hat{y}_{i,t}}{w_{i,t+1}}-\beta_t\hat{\mu}_{i,t}\Big)^\top (\tilde{\mu}_{i,t+1}-\mu)+\epsilon_{\mu_{i,t+1}}^\top(\bar{\mu}_t-\tilde{\mu}_{i,t+1})    \nonumber\\
&=\Big(\epsilon_{\mu_{i,t+1}}-\alpha_t\Big(\frac{\hat{y}_{i,t}}{w_{i,t+1}}-\beta_t\hat{\mu}_{i,t}\Big)\Big)^\top\frac{(\mu_{i,t+1}-w_{i,t+1}\mu)}{w_{i,t+1}}      \nonumber\\
&\hspace{0.1cm}+\alpha_t\Big(\frac{\hat{y}_{i,t}}{w_{i,t+1}}-\beta_t\hat{\mu}_{i,t}\Big)^\top (\tilde{\mu}_{i,t+1}-\mu)+\epsilon_{\mu_{i,t+1}}^\top(\bar{\mu}_t-\tilde{\mu}_{i,t+1})    \nonumber\\
&\hspace{0.1cm}+\Big(\epsilon_{\mu_{i,t+1}}-\alpha_t\Big(\frac{\hat{y}_{i,t}}{w_{i,t+1}}-\beta_t\hat{\mu}_{i,t}\Big)\Big)^\top \frac{(\hat{\mu}_{i,t}-\mu_{i,t+1})}{w_{i,t+1}}.  \label{rp2}
\end{align}

Invoking (\ref{pr1}), (\ref{ag4}), and (\ref{41}) for the last equality of (\ref{rp2}) yields that
\begin{align}
&\epsilon_{\mu_{i,t+1}}^\top(\bar{\mu}_t-\mu)           \nonumber\\
&\leq \frac{\alpha_t}{w_{i,t+1}}\Big(\frac{\hat{y}_{i,t}}{w_{i,t+1}}-\beta_t\hat{\mu}_{i,t}\Big)^\top(\mu_{i,t+1}-\hat{\mu}_{i,t})      \nonumber\\
&\hspace{0.1cm}+\alpha_t\Big(\frac{\hat{y}_{i,t}}{w_{i,t+1}}-\beta_t\hat{\mu}_{i,t}\Big)^\top (\tilde{\mu}_{i,t+1}-\mu)+\epsilon_{\mu_{i,t+1}}^\top(\bar{\mu}_t-\tilde{\mu}_{i,t+1})  \nonumber\\
&\leq \frac{\alpha_t^2}{r}\Big\|\frac{\hat{y}_{i,t}}{w_{i,t+1}}-\beta_t\hat{\mu}_{i,t}\Big\|^2          \nonumber\\
&\hspace{0.4cm}+\alpha_t\Big(\frac{\hat{y}_{i,t}}{w_{i,t+1}}-\beta_t\hat{\mu}_{i,t}\Big)^\top (\tilde{\mu}_{i,t+1}-\mu)            \nonumber\\
&\hspace{0.4cm}+\alpha_t\Big\|\frac{\hat{y}_{i,t}}{w_{i,t+1}}-\beta_t\hat{\mu}_{i,t}\Big\|\|\tilde{\mu}_{i,t+1}-\bar{\mu}_t\|,   \label{rp3}
\end{align}
where the Cauchy-Schwarz inequality and (\ref{pr2}) have been employed in the last inequality.

At this step, substituting (\ref{rp3}) into (\ref{61}) gives rise to
\begin{align}
&\|\bar{\mu}_{t+1}-\mu\|^2       \nonumber\\
&\leq \|\bar{\mu}_t-\mu\|^2+\Big(1+\frac{2}{r}\Big)\frac{\alpha_t^2}{N}\sum_{i=1}^N\Big\|\frac{\hat{y}_{i,t}}{w_{i,t+1}}-\beta_t\hat{\mu}_{i,t}\Big\|^2         \nonumber\\
&\hspace{0.1cm}+\frac{2\alpha_t}{N}\sum_{i=1}^N\frac{\hat{y}_{i,t}^\top}{w_{i,t+1}}(\tilde{\mu}_{i,t+1}-\mu)-\frac{\alpha_t\beta_t}{N}\sum_{i=1}^N 2\hat{\mu}_{i,t}^\top(\tilde{\mu}_{i,t+1}-\mu)            \nonumber\\
&\hspace{0.1cm}+\frac{2\alpha_t}{N}\sum_{i=1}^N\Big\|\frac{\hat{y}_{i,t}}{w_{i,t+1}}-\beta_t\hat{\mu}_{i,t}\Big\|\|\tilde{\mu}_{i,t+1}-\bar{\mu}_t\|. \label{rp4}
\end{align}

For the second term on the right-hand side of (\ref{rp4}), one can conclude that
\begin{align}
\Big\|\frac{\hat{y}_{i,t}}{w_{i,t+1}}-\beta_t\hat{\mu}_{i,t}\Big\|^2&\leq 2\Big\|\frac{\hat{y}_{i,t}}{w_{i,t+1}}\Big\|^2+2\|\beta_t\hat{\mu}_{i,t}\|^2              \nonumber\\
&\leq 2\Big(\frac{B_y}{r}\Big)^2+2\Big(\frac{w_{i,t+1}B_y}{r^2}\Big)^2          \nonumber\\
&\leq \frac{4N^2B_y^2}{r^4},                          \label{67}
\end{align}
where we have employed Lemma \ref{l0} to obtain the first inequality, $w_{i,t+1}\geq r$ along with (\ref{44})-(\ref{45}) for the second inequality, and $r\leq 1$ and $w_{i,t+1}\leq N$ for the last inequality.

For the third term on the right-hand side of (\ref{rp4}), one has that
\begin{align}
&\frac{\hat{y}^\top_{i,t}}{w_{i,t+1}}(\tilde{\mu}_{i,t+1}-\mu)          \nonumber\\
&=(\tilde{y}_{i,t+1}-\bar{y}_t)^\top (\tilde{\mu}_{i,t+1}-\mu)      \nonumber\\
&\hspace{0.4cm}+\bar{y}_t^\top(\tilde{\mu}_{i,t+1}-\bar{\mu}_t)+\bar{y}_t^\top(\bar{\mu}_t-\mu)      \nonumber\\
&\leq \Big(\frac{B_y}{\beta_t r^2}+\|\mu\|\Big)\|\tilde{y}_{i,t+1}-\bar{y}_t\|+B_g\|\tilde{\mu}_{i,t+1}-\bar{\mu}_t\|       \nonumber\\
&\hspace{0.4cm}+\bar{y}_t^\top(\bar{\mu}_t-\mu)              \nonumber\\
&=\Big(\frac{B_y}{\beta_t r^2}+\|\mu\|\Big)\|\tilde{y}_{i,t+1}-\bar{y}_t\|+B_g\|\tilde{\mu}_{i,t+1}-\bar{\mu}_t\|           \nonumber\\
&\hspace{0.4cm}+\frac{1}{N}[L_{t}(x_t,\bar{\mu}_t)-L_{t}(x_t,\mu)],               \label{62}
\end{align}
where we have made use of (\ref{45}) and (\ref{le1}) for getting the inequality, and $\bar{y}_t=g(x_t)/N$ and (\ref{18}) for the last equality.

As for the fourth term on the right-hand side of (\ref{rp4}), we consider the function $h(z):=\|z\|^2$ for $z\in\mathbb{R}^m$, which is convex. Using convexity, one can obtain that
\begin{align}
h(z_1)-h(z_2)\leq \nabla^\top h(z_1)(z_1-z_2),~~~\forall z_1,z_2\in\mathbb{R}^m         \label{63}
\end{align}
which, by letting $z_1=\tilde{\mu}_{i,t+1}$ and $z_2=\mu$, follows that
\begin{align}
\|\tilde{\mu}_{i,t+1}\|^2-\|\mu\|^2\leq 2\tilde{\mu}_{i,t+1}^\top (\tilde{\mu}_{i,t+1}-\mu),         \label{64}
\end{align}
further implying that
\begin{align}
-2\hat{\mu}_{i,t}^\top(\tilde{\mu}_{i,t+1}-\mu)&=-2w_{i,t+1}\tilde{\mu}_{i,t+1}^\top(\tilde{\mu}_{i,t+1}-\mu)        \nonumber\\
&\leq w_{i,t+1}(\|\mu\|^2-\|\tilde{\mu}_{i,t+1}\|^2)   \nonumber\\
&\leq w_{i,t+1}\|\mu\|^2          \nonumber\\
&\leq N\|\mu\|^2,             \label{65}
\end{align}
where $w_{i,t+1}\leq N$ in (\ref{le1}) has been used for obtaining the last inequality.

Now, inserting (\ref{67}), (\ref{62}), and (\ref{65}) into (\ref{rp4}) gives rise to (\ref{57}), which completes the proof.

%%%%%%%%%%%%%%%%%%%%%%%%%%%%%%%%%%%%%%%%%%%%%%%%%%%%%%%%%%%%%%%%%%%%%%%%%%%%%%%%%%%%%%%%%%%%%%%%%%%%%
%\bibliographystyle{amsplain}
%\bibliographystyle{plain}
%\bibliographystyle{unsrt}
%\bibliographystyle{alpha}
%\bibliography{junfen}
%\bibliography{../ref}

%\bibliography{IEEEabrv,ownreference}
%\bibliographystyle{IEEEtran}
%\bibliography{../xiuxian_references}

%%%%%%%%%%%%%%%%%%%%%%%%%%%%%%%%%%%%%%%%%%%%%%%%%%%%%%%%%%%%%%%%%%%%%%%%%%%%%%%%%%%%%%%%%%

\end{document}